%% file: draft.tex
\documentclass[11pt]{amsart}

\usepackage{amsfonts, amsmath, amsthm, amssymb}

\usepackage{comment}
\usepackage[english]{babel} %
\usepackage[T1]{fontenc} %
\usepackage[utf8]{inputenc} %
\usepackage{a4wide} %
\usepackage{subcaption}
\usepackage{bm}
\usepackage{tabularray}

\usepackage{tikz}
\usetikzlibrary{patterns}

\usepackage{graphicx}
\usepackage{fancyhdr}
\usepackage{hyperref}

\usepackage{color}

\definecolor{darkblue}{rgb}{0,0,0.7} 
\definecolor{green}{RGB}{57,181,74} 
\newcommand{\darkblue}{\color{darkblue}} 
\newcommand{\red}{\color{red}} 
\newcommand{\defn}[1]{\emph{\darkblue #1}} 

\newcommand{\NN}{\mathbb{N}}

\newcommand{\arm}{a}
\DeclareMathOperator{\leg}{\ell}
\DeclareMathOperator{\area}{area}
\DeclareMathOperator{\ssim}{sim}
\DeclareMathOperator{\ddef}{def}
\DeclareMathOperator{\dist}{dist}

\newcommand{\SSYT}{\mathrm{SSYT}}

\newcommand{\vlmin}{v_\lambda^{-}}
\newcommand{\vlmax}{v_\lambda^{+}}
\newcommand{\vmin}{v^{-}}

\newcommand{\vlminc}{v^{-} (c, \lambda)}
\newcommand{\vlmaxc}{v^{+} (c, \lambda)}
\newcommand{\vmminc}{v^{-} (c, \mu)}
\newcommand{\vmmaxc}{v^{+} (c, \mu)}

\newcommand{\areap}[2]{\area_{#1} (#2)}
\newcommand{\simp}[2]{\ssim_{#1} (#2)}
\newcommand{\defp}[2]{\ddef_{#1} (#2)}

\newcommand{\arealm}{\areap{\lambda}{\mu}}
\newcommand{\arealt}{\areap{\lambda}{\tau}}
\newcommand{\simlm}{\simp{\lambda}{\mu}}
\newcommand{\Alqt}{A_{\lambda} (q,t)}
\newcommand{\defthm}{\defp{\theta}{\mu}}

\newcommand{\deflm}{\defp{\lambda}{\mu}}
\newcommand{\simthm}{\simp{\theta}{\mu}}

\newcommand{\simtht}{\simp{\theta}{\tau}}
\newcommand{\Athqt}{A_{\theta} (q,t)}
\newcommand{\atm}{a_{\theta}(\mu)}
\newcommand{\imp}[3]{\psi_{#1,#2}(#3)}
\newcommand{\imdal}{\imp{d}{a}{\lambda}}
\newcommand{\uprot}{\overline{r}}
\newcommand{\lowrot}{\underline{r}}
\newcommand{\join}{\wedge}
\newcommand{\meet}{\vee}

\newcommand{\SB}{\bm{S}}
\newcommand{\AB}{\bm{A}}
\DeclareMathOperator{\NewB}{\mathbf{New}}
\newcommand{\ABtwo}{\AB_{m,n}^{|1,2|}}

\newcommand{\Sage}{{\tt SageMath}}

\tikzstyle{rednode} = [red, font=\bf]
\tikzstyle{redline} = [red, thick]

\usepackage{todonotes}

\theoremstyle{plain}
\newtheorem{conjecture}{Conjecture}
\newtheorem{problem}[conjecture]{Problem}
\newtheorem{theorem}{Theorem}[section]
\newtheorem{proposition}[theorem]{Proposition}
\newtheorem{lemma}[theorem]{Lemma}
\newtheorem{corollary}[theorem]{Corollary}

\theoremstyle{definition}
\newtheorem{definition}[theorem]{Definition}
\newtheorem{remark}[theorem]{Remark}

\makeatletter
\renewcommand\paragraph{\@startsection{paragraph}{4}{\z@}{2ex \@plus.5ex \@minus.2ex}{-1em}{\normalfont\normalsize\bfseries}}
\makeatother

\title{Deficit and $(q,t)$-symmetry in triangular partitions}

\author{Loïc Le Mogne}

\author{Viviane Pons}
\address{Universit\'e Paris-Saclay, CNRS, Laboratoire Interdisciplinaire des Sciences du Num\'erique, Orsay, France.}
\email{Loic.lemogne@lisn.fr}
\email{viviane.pons@lisn.fr}

\begin{document}


\maketitle

\begin{abstract}
We study the $(q,t)$-enumeration of triangular Dyck paths considered by Bergeron and Mazin. To do so, we introduce the notion of \emph{triangular} and \emph{sim-sym} tableaux and the \emph{deficit} statistic which is a new interpretation of the \emph{dinv}. We use it to obtain new results and proofs on triangular $2$-partitions and an interesting conjecture for a certain lattice interval $(q,t,r)$-enumeration.
\end{abstract}



\tableofcontents

\section*{Acknowledgements}
The authors thank Olivier Lafitte, the INSMI, the LACIM and the CRM Montréal for the opportunity to work in Montreal on this project, and all the help they gave on administrative issues. We thank François Bergeron for introducing us to the original problem, and for all the very interesting conversations we had during our time in Montreal. Lastly, all computation has been done using SageMath~\cite{SageMath2022}. 

This project has been supported by Project PAGCAP ANR-21-CE48-0020

\section{Introduction}
In~\cite{Berg1}, Bergeron and Mazin study a certain family of partitions that they call \defn{triangular partitions}, motivated by the results of~\cite{BHMPS}. A triangular partition is the maximal partition lying under a given line. The sub-partitions of a triangular partition are called the triangular Dyck paths, generalizing classical and rational Dyck paths. The $(q,t)$-enumeration of paths, such as the $(q,t)$-Catalan numbers~\cite{Hag}, has raised interesting combinatorial questions in recent years, related in particular to representation theory. The Negut formula (see for instance~\cite{Ber22, GHSR20, GN15}) gives a $(q,t)$-enumeration of the sub-partitions of any given partition but the signification of each $(q,t)$ monomial is not clear and the coefficients are not always positive. In a recent work generalizing the \emph{shuffle theorem}~\cite{BHMPS}, the authors found out a combinatorial interpretation when the partition is triangular, using two statistics generalizing the classical \emph{area} and \emph{dinv} of a Dyck path. Indeed, in this case the $(q,t)$ polynomials appear as the coefficients of some symmetric functions and are by nature symmetric themselves. More generally, Conjecture 7.1.1 from~\cite{BHMPS} states that the Negut formula gives positive coefficients if the partition lies under a certain convex curve. As expressed in~\cite{Ber22}, Bergeron further conjectures that this condition is not only sufficient but necessary. Moreover, for each partition $\lambda$, the $(q,t)$-enumeration $A_\lambda(q,t)$ appears to be \emph{Schur positive} in the triangular case~\cite[Conjecture 1]{Berg1} which would relate this combinatorial exploration to representation theory and the study of coinvariant spaces.

In terms of combinatorics, this raises two (very) difficult questions: finding an elementary, combinatorial proof of the $(q,t)$-symmetry for  the enumeration of triangular Dyck paths and proving that they are Schur positive. This includes the classical and rational Dyck path cases, which are special cases of triangular Dyck paths. Besides, another challenge arises related to the work of~\cite{BPR12} on Tamari and $m$-Tamari intervals. Indeed, it is conjectured that a certain $q,t,r$-enumeration of Tamari intervals is also Schur positive and related to trivariate harmonic polynomials. Working on triangular partition, the question is now: is there a generalization of the Tamari lattice that would give a similar symmetric, Schur positive $q,t,r$ enumeration of intervals $A_\lambda(q,t,r)$ ?

These problems are the motivation of the current paper. Our main results are stated in Theorem~\ref{thm:row-regular-athqt}  and Theorem~\ref{thm:qtrsym}: we study the specific (limited) case of $2$-partitions and give a direct proof of the Schur positivity in $2$ and $3$ variables. In particular, we are able to connect the symmetric polynomials $A_\lambda$ introduced by Bergeron to an enumeration of intervals in the $\nu$-Tamari lattices of~\cite{PRV15}. Moreover, we introduce new combinatorial tools such as the \emph{triangular} and \emph{sim-sym} tableaux as well as the \emph{deficit} statistic which explore the combinatorial interpretation of the \emph{dinv} statistic. This led us to state Conjecture~\ref{conj:lattice} which states that our combinatorial interpretation of $A_\lambda$ in terms of intervals of the $\nu$-Tamari lattices could be extended to larger class of triangular partitions.

The paper is organized as follows. Section~\ref{sec:background} recalls the basic notions related to partitions and the definition and main properties of triangular partitions from~\cite{Berg1}. In Section~\ref{sec:triangular}, we define and explore some new combinatorial tools on triangular partitions, namely the \emph{triangular tableau}, the \emph{deficit} statistic and the \emph{sim-sym} tableaux. The main result of this section is Proposition~\ref{prop:deficit} where we prove that the \emph{deficit} statistic actually gives the \emph{dinv}.

Theorem~\ref{thm:row-regular-athqt} is proved in Section~\ref{sec:qt-2parts}. It relies heavily on the previous section and combinatorial tools. In particular, we prove that Schur positive enumeration can be obtained in various combinatorial ways using the notion of sym-sim and row-regular tableaux.

In Section~\ref{sec:lattice}, we explain how this result can be extended to $3$ variables with two major contributions. First we express Conjecture~\ref{conj:lattice} on the combinatorial interpretation of the $A_\lambda$ polynomals using the $\nu$-Tamari lattices in certain cases. We prove this conjecture on $2$-partitions in Theorem~\ref{thm:qtrsym} relying on carefully designed enumerations of Tamari intervals on one hand and semi standard Young tableaux on the other hand.

We share the SageMath code used to run some examples and computations on~\cite{PonSage25}.

\section{Background}
\label{sec:background}

\subsection{Partitions, Young tableaux and Schur functions}


\begin{definition}
A \defn{partition}\footnote{This is actually an integer partition. As there is no ambiguity in the paper, we simply call them partitions.} $\lambda$ of size $n$ is a tuple $\lambda_1 \geq \lambda_2 \geq \dots \lambda_k > 0$ such that \linebreak $|\lambda | := \sum_i^k \lambda_i = n$. For convenience, we consider partitions ending with an infinite number of $0$-parts, $\lambda_j := 0$ for $j > k$.  A \defn{$k$-partition} is a partition $\lambda = (\lambda_1, \dots, \lambda_{k'})$ with $k' \leq k$. We say that a partition $\mu$ is a \defn{sub-partition} of another partition $\lambda$ if $\mu_i \leq \lambda_i$ for all $i$.
\end{definition}

As usual, we represent partitions by their \defn{Ferrers diagram} (French style) such that each value $\lambda_i$, from bottom to top, corresponds to a line of $i$ square cells as shown on Figure~\ref{fig:young}. We identify each cell by a tuple $(\ell, col)$ where $\ell$ is its line (starting at 0) while $col$ is its column. This is consistent with the implementation of the Ferrers diagram in \Sage~and thus with our demo worksheet~\cite{PonSage25}. For example, on Figure~\ref{fig:young}, the bottom left-most cell is $(0,0)$, while the ending cell of the second row is $(1,2)$.

\begin{definition}
A \defn{Young tableau} on a partition $\lambda$ is an integer filling of the Ferrers diagram of~$\lambda$ (we place a number in each cell). We consider a tableau $\theta$ as a function from the cells of $\lambda$ to $\NN$ and write $\theta(c)$ the value of the cell $c$ in $\theta$. Besides, we rite $\theta_i$ the number of cells $c$ such that $\theta(c) = i$.

A \defn{semi-standard Young tableau} on a partition $\lambda$ is a tableau where numbers are strictly increasing along each column and largely increasing along each row. The set of semi-standard Young Tableaux over $\lambda$ is written \defn{SSYT($\lambda$)}.

A \defn{standard Young tableau} is a tableau on a partition $\lambda$ such that numbers are strictly increasing along each column and line.
\end{definition}


We present examples of partitions and tableaux in Figures~\ref{fig:young} and~\ref{fig:tableau}.





\begin{figure}[ht]
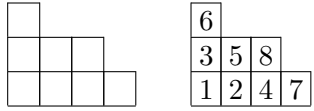

    \centering
    \begin{tabular}{cc}
    \input{figures/partition.tex}
    &
    \input{figures/Young_tableau.tex}
    \end{tabular}
    
    \caption{The Ferrers diagram of $(4,3,1)$ and an example of standard Young tableau.}
    \label{fig:young}
\end{figure}


\begin{figure}[ht]
    \centering
    \input{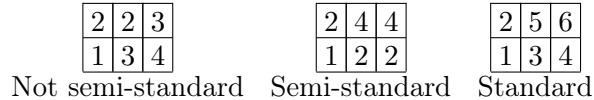}
    \caption{Some Young tableaux for the partition $(3,3)$.}
    \label{fig:tableau}
\end{figure}

\defn{Schur functions} are infinite series defined on an infinite alphabet $x_1, x_2, \dots$. They form a famous basis of symmetric functions~\cite{Mcdo}. There exists many equivalent definitions, we use the following. For a partition $\lambda$, the Schur function $s_{\lambda}$ is given by

\begin{equation}
\label{eq:schur}
s_{\lambda} := \sum_{\theta \in \SSYT(\lambda)}\prod_{i \in \NN} x_{i}^{\theta_{i}} 
\end{equation}

Each term of the sum is given by a semi-standard Young tableau of the partition. For example, the (middle) semi-standard Young tableau of Figure~\ref{fig:tableau} gives the monomial $x_1 x_2^3 x_4^2$. As there are infinitely many semi-standard Young tableaux for a given partition, a Schur function has infinitely many terms. In practice, Schur functions can be developed on finite number of variables by setting all other variables to~$0$. For example, we have 

\begin{align}
s_{3,1}(q,t) &= q^{3}t + q^{2}t^{2} + qt^{3}, \\
s_{2,1}(q,t,r) &= q^{2} r + q r^{2} + q^{2} t + 2 \, q r t + r^{2} t + q t^{2} + r t^{2}, \\
s_{1,1,1}(q,t,r) &= qtr.
\end{align}
These are called \defn{Schur polynomials}. Figure~\ref{fig:schur-example} illustrates the correspondence with semi-standard Young tableaux.

\begin{figure}[ht]
\input{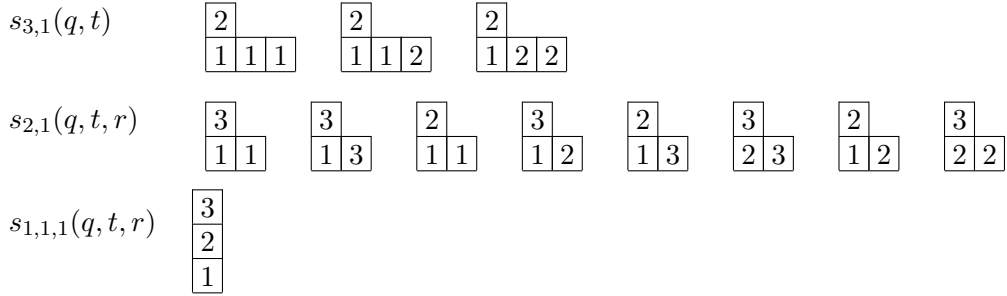}
\caption{Semi-standard Young tableaux corresponding to some Schur polynomials.}
\label{fig:schur-example}
\end{figure}

\subsection{Triangular partitions}
\label{subsec:triangular}



\begin{definition}[From~\cite{Berg1}]
A \defn{triangular partition} $\lambda$ is a partition such that there exist two real numbers $r$ and $s$ with $\lambda_{j} = \lfloor r - \frac{jr}{s} \rfloor$ for all $j \leq s$, and $\lambda_j = 0$ otherwise.
\end{definition}

We say that the line joining $(r,0)$ and $(0,s)$ is an~\defn{$r$-$s$-line}. Then $\lambda$ is the greatest partition lying under the $r$-$s$-line. We also say that this line \defn{cuts off} $\lambda$ . See an example on Figure~\ref{fig:triang}.


\begin{figure}[ht]
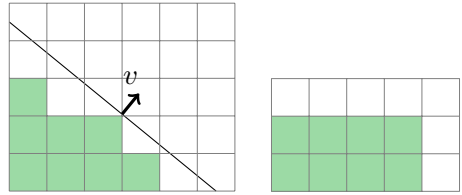

    \centering
    \begin{tabular}{cc}
    \input{figures/triangular_partition.tex}
    &
    \input{figures/non_triangular.tex}
    \end{tabular}
    \caption{Left: the triangular partition $(4,3,1)$ with one line that cuts it off and its slope vector. Right: a non-triangular partition $(4,4)$.}
    \label{fig:triang}
\end{figure}

When there exists an $r$-$s$-line cutting off $\lambda$ with $r$ and $s$ integer values, $\lambda$ is said to be a \defn{rectangular} partition. It is further said to be \defn{rational} when $r$ and $s$ are coprime. Later on, it will be important to characterize the \defn{slope} of the line. In accordance with~\cite{Berg1}, this can be done through a single parameter $v$ which characterizes a \emph{slope vector}.

\begin{definition}
The \defn{slope vector} of an $r$-$s$-line is the vector orthogonal to the line and of coordinates $(v,1-v)$. For simplicity, we write ``the slope $v$''. We say that $\lambda$ admits a slope $v$ if it is cut off by an $r$-$s$-line of slope vector $(v, 1-v)$.
\end{definition}

A triangular partition admits an infinite number of slope vectors which corresponds to certain open interval for the value of $v$. Bergeron and Mazin~\cite{Berg1} give a simple way to compute the slope interval using the hooks of the partition.

\begin{definition}
Let $c$ be a cell of $\lambda$, the \defn{arm} of $c$, written $\arm(c)$, is the number of cells lying to the right of $c$ in $\lambda$ and the \defn{leg} of $c$, written $\leg(c)$, is the number of cells lying above $c$ in $\lambda$. Then 

\begin{align}
\label{eq:vl}
\vlminc &:= \frac{\leg(c)}{\arm(c) + \leg(c) + 1}, &
\vlmaxc &:= \frac{\leg(c) + 1}{\arm(c) + \leg(c) + 1}.
\end{align}

Then $]\vlminc, ~\vlmaxc[$ is the \defn{admissible interval} of the cell $c$. Besides, we write

\begin{align}
\vlmin &:= \max_{c \in \lambda}(\vlminc), &
\vlmax &:= \min_{c \in \lambda}(\vlmaxc).
\end{align}

\end{definition}

Bergeron and Mazin prove that a partition is triangular if and only if the intersection of all admissible intervals given by  $]\vlmin, \vlmax[$ is not empty and this also characterizes all possible slopes for the partition.

\begin{lemma}[Lemma 1.2 of \cite{Berg1}]
\label{lem:triang}
Let $\lambda$ be a partition, then $\lambda$ is triangular if and only if $\vlmin < \vlmax$. Then for all $v \in ]\vlmin, \vlmax[$, $\lambda$ is cut off by a line of slope vector $v$.
\end{lemma}

For example, on the triangular partition of Figure~\ref{fig:triang}, for $c$ the cell $(0,0)$, we have $\leg(c) = 2$, $\arm(c) = 3$ which gives $\vlminc = \frac{1}{3}$ and $\vlmaxc = \frac{1}{2}$. By computing the values for the other cells, we actually find that $\vlmin = \frac{1}{3}$ and $\vlmax = \frac{1}{2}$. As $\frac{1}{3} < \frac{1}{2}$ this is indeed a triangular partition. On the other hand, for the second partition of Figure~\ref{fig:triang}, for $c = (0,3)$ the last cell of the lowest row, we find $\arm(c) = 0$, $\leg(c) = 1$ and $\vlminc = \frac{1}{2}$ whereas for $c' = (1,0)$ ,the first cell of the second row, we have $\leg(c') = 0$, $\arm(c') = 3$, and $v^{+}(c',\lambda) = \frac{1}{4}$. This is not a triangular partition.

In~\cite{PonSage25}, we provide the code to generate all triangular partitions of a given size and to test if a given partition is triangular.

\subsection{Area, sim and (q,t)-enumeration}
In this paper, we are interested in the $(q,t)$ enumeration of sub-partitions of a given triangular partition $\lambda$. Following the terminology of~\cite{Berg1}, we call such objects \defn{triangular Dyck paths}.

\begin{definition}
A \defn{triangular Dyck path} is a tuple $(\lambda, \mu)$ such that $\mu$ is a sub-partition of $\lambda$ and $\lambda$ is a triangular partition. The \defn{area} of a triangular Dyck path is the number of cells belonging to $\lambda$ but not $\mu$, given by $\arealm:= \sum_{i \in \NN} \lambda_i - \mu_i$.
\end{definition}

Note that we do not require that $\mu$ is also triangular. We show an example on Figure~\ref{fig:Area} where $\lambda = (7,6,4,3,1)$ and $\mu = (5,5,3,2)$. We draw the path between the two partitions in thick red. The area is made of cells in $\lambda \setminus \mu$, \emph{i.e.}, the cells which are outside the path (in light red in the example). Here we have $\arealm = 6$.

When $\lambda$ is the staircase partition, \emph{i.e.}, $\lambda = (n-1, n-2, \dots, 1)$, then the triangular Dyck paths $(\lambda, \mu)$ correspond to classical Dyck paths of length $n$ : start with an up-step then read the path from top to bottom transforming each vertical step into an up-step and each horizontal step with a down-step, and finish with a down-step. See Figure~\ref{fig:staircase21} for an example. Similarly, when $\lambda$ is a rational partition, then its sub-partitions are known as rational Dyck paths. 


\begin{figure}[ht]
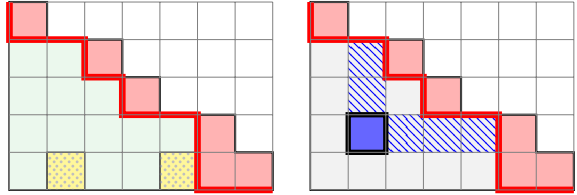

    \centering
    \begin{tabular}{cc}
    \scalebox{.5}{\input{figures/sim.tex}} &
    \scalebox{.5}{\input{figures/hook.tex}} 
    \end{tabular}
    \caption{Left: triangular Dyck path $(7,6,4,3,1), (5,5,3,2)$ with ${\arealm = 6}$ (cells in light red), $\simlm = 13$ (cells in light green), and $\deflm = 2$ (dotted yellow cells). Right: the same triangular Dyck path where the hook of a given cell is shown.}
    
    \label{fig:Area}
\end{figure}

\begin{figure}[ht]
\input{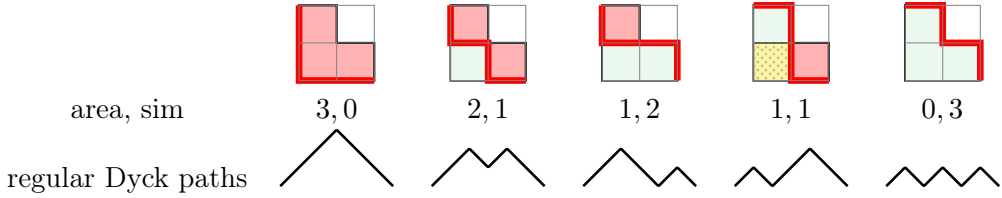}
\caption{The 5 triangular Dyck paths of the partition $(2,1)$ with their area and sim.}
\label{fig:staircase21}
\end{figure}

\begin{definition}[From~\cite{Berg1}, Section 4.1]\label{def:defcsim}
Let $(\lambda, \mu)$ be a triangular Dyck path. A cell $c$ of~$\mu$ is said to be \defn{similar} if
\begin{equation}\label{eq:eqcsim}
\vmminc < \frac{\vlmin + \vlmax}{2} \leq \vmmaxc.
\end{equation}
In other words, the \defn{mean slope} of the slope interval of $\lambda$ belongs to the semi-closed admissible interval of the cell $c$ in $\mu$. Then $\simlm$ is the number of similar cells in the triangular Dyck path $(\lambda, \mu)$.
\end{definition}

\begin{remark}
Note that in~\cite{Berg1}, the large and strict inequalities are reversed. Both choices are arbitrary and give similar results (as long as there is one large and one strick inequality). In this this paper, placing the large inequality in second position allows for a better compatibility later with the Tamari lattices.
\end{remark}

For example on the right of Figure~\ref{fig:Area}, we have selected a cell whose arm is $a(c) = 3$ and leg is $\ell(c) = 2$. The mean slope of $\lambda$, $\frac{45}{112}$ belongs to the admissible interval of the selected cell as we obtain $\vmminc = \frac{1}{3}$ and $\vmmaxc = \frac{1}{2}$ (from definitions~\eqref{eq:vl}). By checking all cells, we obtain $\simlm = 13$: the similar cells are colored in light green while the non similar cells are in dotted yellow. In~\cite{PonSage25} you find an implementation of triangular Dyck paths with the computation of the \emph{sim} statistic.

It is important to understand that the similarity of a cell depends on both partitions of the triangular Dyck path. Indeed, the mean slope is computed on $\lambda$ but the admissible cell interval is computed on $\mu$. If $(\lambda, \mu)$ is a triangular Dyck path such that all the cells of $\mu$ are similar, then in particular $\mu$ is triangular: a slight perturbation of the mean slope of $\lambda$ will always cut off $\mu$ (see~\cite[Lemma 4.1]{Berg1}). 
When all cells of $\mu$ are similar cells, we say that $(\lambda, \mu)$ is a \defn{mean-similar} triangular Dyck path.


The \emph{sim} statistic already appears in~\cite{BHMPS} and is a generalization of the famous \emph{dinv} statistic that can be found in~\cite{Hag} for the classical $(q,t)$-Catalan enumeration. Indeed, when~$\lambda$ is the staircase partition, then the \emph{sim} of $(\lambda, \mu)$ is actually the \emph{dinv} of the corresponding Dyck path as you can check on Figure~\ref{fig:staircase21}.




%
%

We can now define the $(q,t)$-polynomials associated to a triangular partition $\lambda$ as:

\begin{equation}
\label{eq:alqt}
\Alqt = \sum_{\mu \subseteq \lambda}q^{\arealm} t^{\simlm}.
\end{equation}

These polynomials have been shown to be symmetric in $q$ and $t$ via an algebraic proof~\cite{BHMPS}. However, the question of finding a combinatorial proof is still open even in the classical case. For example, for the partition $\lambda = (3,2)$ whose triangular Dyck paths are shown on Figure~\ref{fig:tdp3_2}, we have the following:

\begin{equation}
A_{3,2}(q,t) = q^{5} + q^{4} t + q^{3} t^{2} + q^{2} t^{3} + q t^{4} + t^{5} + q^{3} t + q^{2} t^{2} + q t^{3}
\end{equation}

\begin{figure}[ht]
\input{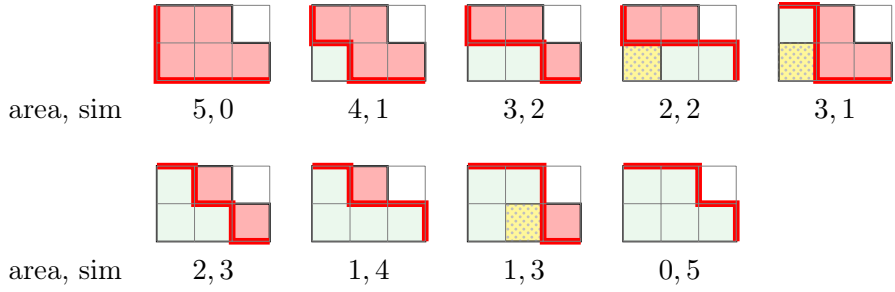}
\caption{The triangular Dyck paths of $\lambda = (3,2)$.}
\label{fig:tdp3_2}
\end{figure}


As the $\Alqt$ polynomials are symmetric, they can be expressed in terms of Schur functions with the following conjecture.

\begin{conjecture}[Conjecture 1 of~\cite{Berg1}]
\label{conj:schur}
The polynomials $\Alqt$ are Schur-positive.

\end{conjecture}


Here are some examples of computation, more can be found in~\cite{Berg1} or computed using our demo worksheet~\cite{PonSage25} where we check the conjecture for triangular partitions up to size~$10$.

\begin{align}
\label{eq:qtschur_a321}
A_{3,2,1}(q,t) &= s_{6}(q,t) + s_{4,1}(q,t) + s_{3,1}(q,t), \\
A_{3,2}(q,t) &= s_{5}(q,t) + s_{3,1}(q,t), \\
A_{2,2,1}(q,t) &= s_{5}(q,t) + s_{3,1}(q,t)
\end{align}

\section{Combinatorial interpretation of similar cells}
\label{sec:triangular}

In this section, we propose a new combinatorial interpretation for similar cells which is easier to use and leads to new results. The first step is to define a certain Young tableau associated to each triangular partition.

\subsection{Triangular Tableau and deficit}

Bergeron and Mazin define the \defn{ray} of a given slope vector $(v, 1-v)$ by the set of triangular partitions which admits $(v,1-v)$ as a slope vector~\cite[Formula~(3.1)]{Berg1}. They give the following Lemma.

\begin{lemma}[Lemma~4.1 of~\cite{Berg1}]
Let $\lambda$ be a triangular partition, then there exists $\epsilon > 0$ irrational, such that the set of mean-similar triangular Dyck paths $(\lambda, \mu)$ are exactly those that lie on the ray corresponding to the slope $t_\lambda + \epsilon$.
\end{lemma}

Besides, they prove that the triangular partitions of a given ray form an infinite chain in the Young lattice restricted to triangular partitions. In particular, if $v$ is an \emph{irrational slope} and $R_v$ its ray, then for each $n \in \NN$, there is a unique triangular partition of size $n$ in $R_v$ and it is obtained by adding a well chosen box on the unique triangular partition $\mu$  of size $n-1$ in $R_v$. As a direct consequence, we obtain

\begin{proposition}
    Let $\lambda$ be a triangular partition. Then for all $0 \leq k \leq |\lambda |$, there exists a unique mean-similar triangular Dyck path $(\lambda, \mu)$ of area $k$. Besides, if $(\lambda, \alpha)$ and $(\lambda, \mu)$ are two mean-similar triangular Dyck paths with $|\alpha| \leq |\mu |$, then $\alpha \subseteq \mu$.
\end{proposition}

%
In other words, if $(\lambda, \mu)$ is such that all cells of $\mu$ are similar cells, then there is a unique way to remove a cell from $\mu$ such that all remaining cells stay similar. We illustrate this on Figure~\ref{fig:triangular_32}. Geometrically, the mean-slope of $\lambda$ can always be slightly shifted to obtain an irrational slope. By \emph{moving} this irrational slope towards the origin, we  ``touch'' the cells of $\lambda$ one by one, which gives the removal order on the cells of the sub-partitions in order to stay mean-similar, see Figure~\ref{fig:TYT} for an example. This leads to the following definition.

\begin{definition}
    \label{def:tyt}
    The \defn{Triangular Young tableau} of a triangular partition $\lambda$ with $|\lambda| = n$ is the unique standard Young tableau such that for all $0 \leq k \leq n$, the cells $1$ to $k$ form a sub-partition~$\mu$ with $(\lambda, \mu)$ a mean-similar triangular Dyck path.
\end{definition}

\begin{figure}[ht]
\input{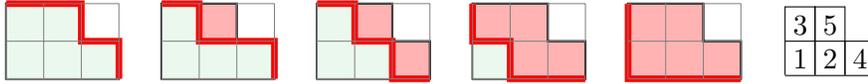}
\caption{The mean-similar triangular Dyck paths of $\lambda = (3,2)$ and the corresponding triangular Young tableau.}
\label{fig:triangular_32}
\end{figure}



\begin{figure}[ht]
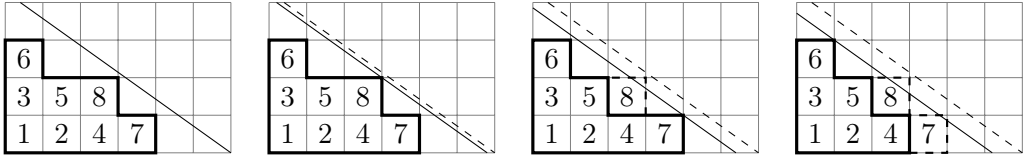

    \centering
    \begin{tabular}{cccc}
    \input{figures/triangular_tableau.tex}
    &
    \input{figures/triangular_tableau2.tex}
    &
    \input{figures/triangular_tableau3.tex}
    & 
    \input{figures/triangular_tableau4.tex}
    \end{tabular}
    \caption{The triangular Young tableaux of $(4,3,1)$ with its slope construction.}
    \label{fig:TYT}
\end{figure}
    
%
%
%

\subsection{Deficit of a sub-partition}

\begin{definition}
    Let $(\lambda, \mu)$ be a triangular Dyck path and $\theta$ a standard Young tableau of shape~$\lambda$.  We say that there is a \defn{$\theta$-inversion} in $(\lambda, \mu)$ if there exists a cell $c_1 = (i_1,j_1)$ in $\mu$ and a cell $c_2 = (i_2,j_2)$ in $\lambda \setminus \mu$ such that $\theta(c_1) > \theta(c_2)$. In particular, $c_1$ cannot be in the same line or column than $c_2$. Then we say that that the cell $c = (\min(i_1,i_2), \min(j_1,j_2))$, which is at the hook of $c_1$ and $c_2$, is a \defn{$\theta$-deficit cell} and we call \defn{deficit} of the triangular Dyck path $(\lambda, \mu)$ with tableau~$\theta$, written $\defthm$, the number of $\theta$-deficit cells.
    \end{definition}
    
    \begin{figure}[ht]
        \centering
        \input{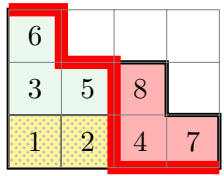}
        \caption{A triangular dyck path on $(4,3,1)$ with two deficit cell: 1 and 2.}
        \label{fig:DefC}
    \end{figure}
    
    See Figure~\ref{fig:DefC} for an example: $1$ is a deficit cell because it is at the hook of~$6$ and~$4$ and~$2$ is a deficit cell because it is at the hook of $5$ and $7$. Note that it can happen that multiple $\theta$-inversions give the same deficit cell. Given their definition, $\theta$-deficit cells have some clear combinatorial properties, we express one of them in the following Lemma.
    
\begin{lemma}
\label{lem:def-interior}
Let $c$ be a $\theta$-deficit cell, then $c$ is in the \emph{interior} of the partition $\lambda$, \emph{i.e.}, $\arm(c) > 0$ and $\leg(c) > 0$.
\end{lemma}

\begin{proof}
The cell $c$ lies at the hook of $c_1$ and $c_2$. Besides $c_1$ and $c_2$ cannot be on the same line or column, which means $c_1 \neq c$ and $c_2 \neq c$: the cell $c$ must have a positive number of cells lying above it and to its right. 
\end{proof}

     We prove the following which gives a more combinatorial approach to the computation of similar cells.
    

    \begin{proposition}
    \label{prop:deficit}
    Let $\lambda$ be a triangular partition and $\theta$ its triangular Young tableau as in Definition~\ref{def:tyt}. Then for any triangular Dyck path $(\lambda, \mu)$, the $\theta$-deficit cells of $(\lambda, \mu)$ are exactly the non-similar cells of $\mu$. In particular, we have $|\lambda| = \arealm + \simlm + \defthm$.
    \end{proposition}
    
    \begin{proof}
        We first show that non-$\theta$-deficit cells of $(\lambda,\mu)$ are similar cells of $(\lambda,\mu)$. 
        
        
        Let $d$ be a non-$\theta$-deficit cell. We choose $c$ in the hook of $d$ in $\mu$ such that $\theta(c)$ is maximal. As $d$ is not a deficit, all cell $c'$ in the hook of $d$ in $\lambda \backslash \mu$ are such that $\theta(c') > \theta(c)$. We now consider the mean-similar partition $\alpha$ of size $|\alpha|=\theta(c)$, \emph{i.e.}, the cells of $\alpha$ are all the cells~$x$ of~$\lambda$ such that $\theta(x) \leq \theta(c)$. In particular, all the cells of the hook of $d$ in $\mu$ are also in $\alpha$ whereas all the cells of of the hook of $d$ in $\lambda \backslash \mu$ are not:  the hook of d in $\alpha$ is the same than the hook of d in $\mu$. Furthermore, $d$ is a similar cell in $(\lambda,\alpha)$ by definition of a mean-similar partition, therefore $d$ is similar in $(\lambda, \mu)$.
    
        Let us now take $d$ a $\theta$-deficit cell. Without loss of generality, we can consider that there is $c \in \mu$ in the arm of $d$ and $c' \in \lambda \backslash \mu$ in the leg of $d$ such that $\theta(c) > \theta(c')$ (we can take the $c$ as the last cell in the arm of $d$ in $\mu$).
        
        If we consider $\alpha$ the mean-similar partition of size $\theta(c)$, the hook of $d$ in $\mu$ is striclty contained in the hook of $d$ in $\alpha$ ($c'$ is in the hook of $d$ in $\alpha$ but not in $\mu$). We write $a$ the length of the arm of $c$ in $\alpha$ and $l$ the length of its leg. Because $\alpha \backslash c$  is still mean-similar, we have that $c$ in $\alpha \backslash c$ is similar. As such, we have: \[ \frac{l}{(a-1)+l+1} < \frac{\vlmin + \vlmax}{2} + \epsilon < \frac{l+1}{(a-1)+l+1} \]
    
        However, in $\mu$, we have $l_{\mu}(d) < l$ and $ a_{\mu}(d) = a$ (and as such $\frac{l_{\mu}(d)+1}{a_{\mu}(d) + l_{\mu}(d)+1} \leq \frac{l}{a+l}$). Therefore, we have: \[ \frac{l_{\mu}(d)}{a_{\mu}(d) + l_{\mu}(d)+1}  < \frac{l_{\mu}(d)+1}{a_{\mu}(d) + l_{\mu}(d)+1} < \frac{\vlmin + \vlmax}{2} + \epsilon\]
    
        This means that $d$ is not similar.
    \end{proof}
    
    In the rest of the paper, we write $\deflm := \defthm$ of a triangular Dyck path where $\theta$ is the triangular Young tableau.

    
    For the staircase partition (which gives classical Dyck paths), the triangular tableau is the \defn{top-down} tableau where we label the corners of the partition in decreasing order from top to bottom as shown in Figure~\ref{fig:staircase}. The deficit then gives a new interpretation for the classical \emph{dinv} statistic.
    
    \begin{figure}[ht]
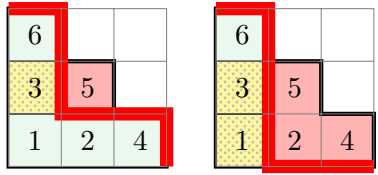

        \centering
    \begin{tabular}{cc}
    \input{figures/dinv_area1.tex}
    &
    \input{figures/dinv_area2.tex}
    \end{tabular}
    \caption{Two example of triangular Dyck paths with $\lambda =(3,2,1)$ the staircase partition. On the left: we have $\mu = (3,1,1)$ and $\simlm = 4$ ($3$ is a deficit cell). On the right, we have $\mu = (1,1,1)$ and $\simlm = 1$ ($1$ and $3$ are deficit cells).}
    \label{fig:staircase}
    \end{figure}

\subsection{Sim-Sym Tableau}

\begin{figure}[ht]
\input{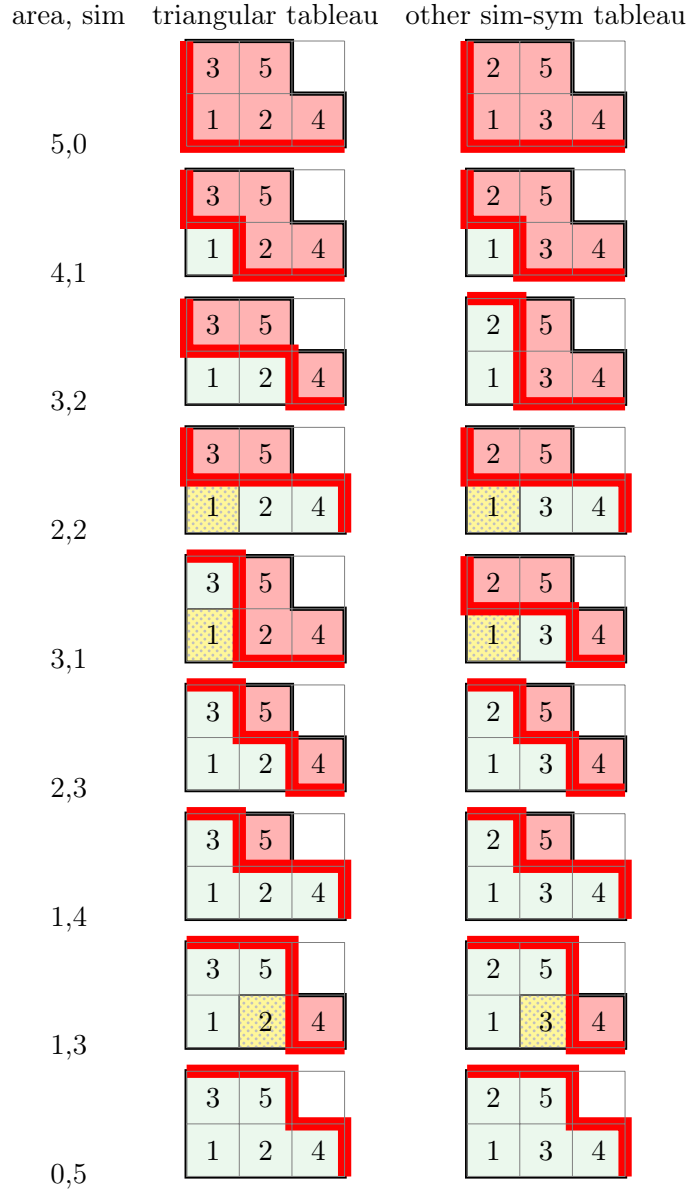}
\caption{The triangular Dyck paths of the partition $(3,2)$ with their $\theta$-similar cells for the two possible sym-sim tableaux.}
\label{fig:sim-sym-32}
\end{figure}

The deficit statistic is defined for any tableau $\theta$ and just happens to coincide with the non-similar cells when $\theta$ is the triangular Young tableau of $\lambda$. It is then natural to extend the definition of the \emph{sim} statistic to any Young tableau.

\begin{definition}
Let $\theta$ be a standard Young tableau of triangular shape $\lambda$. Then for any triangular Dyck paths $(\lambda, \mu)$, we say that a cell of $\mu$ is \defn{$\theta$-similar} if it is not a $\theta$-deficit cell. We write $\simthm$ the number of $\theta$-similar cells. 
\end{definition}


With that definition, again it is natural to extend the definition of~\eqref{eq:alqt} to this new setting, we define

\begin{equation}
\Athqt = \sum_{\mu \subseteq \lambda} q^{\arealm} t^{\simthm}
\end{equation}
where $\theta$ is a standard Young tableau of shape $\lambda$. These polynomials are not symmetric in general. On the other hand, the triangular Young tableau is not the only tableau giving the symmetry. This led us to introduce the following objects which we believe are of high interest when studying the $(q,t)$ symmetry.

\begin{definition}
A \defn{sim-sym} (similar-symmetric) tableau  is a Young tableau $\theta$ of triangular shape $\lambda$ such that $\Athqt = \Alqt$.
\end{definition}

Note that, in theory we could have $\Athqt \neq \Alqt$ and still be symmetric. From experimentation, we conjecture that it never happens.


It is important to note that this equality does not necessarily mean that $\simthm = \simlm$ for all triangular Dyck paths, only that we have a bijection over the triangular Dyck paths of $\lambda$ which keeps the \emph{area} constant while sending the $\theta$-\emph{sim} statistic to the classical \emph{sim}. Figure~\ref{fig:sim-sym-32} gives the two possible sym-sim tableaux for the partition $(3,2)$ along with the list of triangular Dyck paths with their $\theta$-similar cells. Other tableaux of $(3,2)$ are not sim-sym. For example, for the following tableau
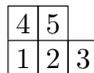
\begin{center}
{\def\lr#1{\multicolumn{1}{|@{\hspace{.6ex}}c@{\hspace{.6ex}}|}{\raisebox{-.3ex}{$#1$}}}
\raisebox{-.6ex}{$\begin{array}[t]{*{3}c}\cline{1-2}
\lr{4}&\lr{5}\\\cline{1-3}
\lr{1}&\lr{2}&\lr{3}\\\cline{1-3}
\end{array}$}
}
\end{center}

we get, the polynomial

$$ q^{5} + q^{4} t + q^{3} t^{2} + q^{2} t^{3} + q t^{4} + t^{5} + q^{3} t + q^{2} t^{2} + q t^{2}$$
which is not symmetric. The number of sim-sym tableaux seem to be rather small compared to the total number of tableaux. For example, for the partition $(4,2,1)$, we find $7$ sim-sym tableaux (see Figure~\ref{fig:sim-sym-432}) among the $35$ standard Young tableaux.

\begin{figure}[ht]
\input{figures/sim-sym-432}
\caption{The sim-sym tableaux of $(4,3,2)$}
\label{fig:sim-sym-432}
\end{figure}



Characterizing sim-sym tableaux appears to be an interesting but difficult question. One way to find different sim-sym tableaux is to choose a different slope to sweep down the partition. Indeed, the triangular Young tableau is defined using a slight deformation of the mean slope (to ensure that it is irrational). But any irrational slope would actually give a tableau and computational exploration suggests that they are always sim-sym tableaux. Nevertheless, they are not the only examples. For example, for the partition $(4,2,1)$, we find only two slope-tableaux which are the first two tableaux of Figure~\ref{fig:sim-sym-432} (the first one is the mean-slope triangular tableau), but there are actually 7 sim-sym tableaux. In the rest of the paper, we characterize sim-sym tableaux on $2$-partitions and explore some interesting conjectures related to the lattice structure of triangular Dyck-paths and sim-sym tableaux.

All examples have been computed using \Sage{} and can be found in our demo worksheet~\cite{PonSage25}.

\section{Sim-sym tableaux and $(q,t)$ enumeration on $2$-partitions}
\label{sec:qt-2parts}

Proving Conjecture~\ref{conj:schur} in general is difficult. In this section, we restrict our study to the case of $2$-partitions (partitions with only $2$ parts).

\subsection{Characterization of triangular $2$-partitions}

We start with a simple property of $2$-partitions.

\begin{proposition}
\label{prop:triang}
Let $\lambda=(m,n)$ be a 2-partition, $\lambda$ is triangular if and only if $n \leq \lceil \frac{m}{2} \rceil$.
\end{proposition}

\begin{remark}
\label{rem:triang}
By simple computation, the condition  $n \leq \lceil \frac{m}{2} \rceil$ is equivalent to $n \leq m - n + 1$.
\end{remark}

To prove the proposition, we use the following Lemma.

\begin{lemma}
\label{lem:vl2parts}
Let $\lambda = (m,n)$ be a partition of length $2$. Then

\begin{align}
\vlmin &= \frac{1}{m-n +2}, &\vlmax &= \min \left( \frac{2}{m+1}, \frac{1}{n} \right).
\end{align}

\end{lemma}

\begin{proof}
The maximal value of $\vlminc$ is obtained for $c = (0,n-1)$ the $n^{th}$ cell of the first row. Indeed, for all cells $c'$ on the second row or right of $c$, we have $\leg(c') = 0$ and so $\vmin(c', \lambda) = 0$ and then $c$ is the cell maximizing the value $\arm(c) + \leg(c) + 1$. We have $\arm(c) = m-n$ and $\leg(c) = 1$, hence the result.

The minimal value of $\vlmaxc$ is obtained either on $c = (0,0)$, the first cell of the first row, or on $c = (1,0)$,  the first cell of the second row. We get $\arm(0,0) = m-1$, $\leg(0,0) = 1$, $\arm(1,0) = n-1$, and $\leg(1,0) = 0$. This gives the result.
\end{proof}

\begin{proof}[Proof of Proposition~\ref{prop:triang}]
We use the characterization of triangular partitions given in Lemma~\ref{lem:triang} from Bergeron and Mazin.

If $n > \lceil \frac{m}{2} \rceil$, in particular $n > \frac{m+1}{2}$ and $\vlmax = \frac{1}{n}$. As $m = \lfloor \frac{m}{2} \rfloor + \lceil \frac{m}{2} \rceil$ we find $m-n < \lfloor \frac{m}{2} \rfloor \leq \lceil \frac{m}{2} \rceil < n$ which gives $m-n +2 \leq n$ and so $\vlmax \leq \vlmin$: the partition is not triangular. 

On the other hand, if $n \leq \frac{m}{2}$, we have $m-n+2 \geq \lfloor \frac{m}{2} \rfloor + 2$. And as we have $\lfloor \frac{m}{2} \rfloor \geq \lceil \frac{m}{2} \rceil -1$ as well as $\lfloor \frac{m}{2} \rfloor \geq \frac{m+1}{2} -1$. We obtain $m-n+2 \geq \lceil \frac{m}{2} \rceil +1 > n$ and $m+n+2 > \frac{m+1}{2}$ which gives $\vlmin < \vlmax$.
\end{proof}

\begin{remark}
The result can also be seen geometrically from the original definition of triangular partition. 
 If $\lambda = (m,n)$ is triangular, then there is a line that cuts off $\lambda$. This line necessarily passes strictly under the points $(1,3)$, and $(m+1,1)$. It then passes under the middle point $(\frac{m}{2} +1, 2)$ which gives $n \leq \lceil \frac{m}{2} \rceil$. Besides, by slightly translating the line passing through $(1,3)$ and $(m+1,1)$ and rotating it, we see that we can obtain all the wanted partitions.
 
\end{remark}

\subsection{Row-regulars tableaux}

\begin{definition}
\label{def:row-regular}
Let $\lambda = (m,n)$ be a triangular partition. We say that a tableau $\theta$ on $\lambda$ is \defn{row-regular} or, more precisely that it is the $i$-row-regular tableau of $\lambda$, if the labels on the upper line of $\lambda$ are given by $n+i, n+i+2, n+i+4, \dots, n+i+2(n-1)$ with $1 \leq i \leq m - 2(n-1)$.

As $\lambda$ only has two (increasing) rows, this entirely characterizes $\theta$.

The \defn{minimal} row-regular tableau of $\lambda$ is the row-regular tableau with $i=1$, the \defn{maximal} is the row-regular with $i = m - 2(n-1)$ and we call the row-regular tableau with $i = m -2n -3$ (when it exists) the \defn{pre-maximal} row-regular tableau of $\lambda$.
\end{definition}

In particular, there are $m - 2(n-1)$ row-regulars tableaux for $\lambda = (m,n)$ and the values $1,2,\dots,n$ are always placed on the first row. On Figure~\ref{fig:row-regular}, we present all row-regulars tableaux of the partition $(9,4)$. The values of the upper row increase two by two, forcing a certain ``alternation'' between the two rows. If $\theta$ is not the maximal row-regular tableau, the values of the lower row are always $1,2,\dots,n+i-1$, then $n+i+1, n+i+3, \dots, n+i+2(n-1)+1$, then all remaining values up to $n+m$.

\begin{figure}[ht]
\input{figures/row-regular}
\caption{The $3$ possible row-regulars tableaux of the partition $(9,4)$.}
\label{fig:row-regular}
\end{figure}

\begin{lemma}
\label{lem;max-row-reg}
The maximal row-regular tableau is the only one where the label $m+n$ is placed on the upper row.
\end{lemma}

\begin{proof}
This is direct from the definition. The maximal label of the upper row is given by $\ell = n+i+2(n-1)$. If $\theta$ is the maximal row-regular tableau, then $i = m - 2(n-1)$ and $\ell = n + m$. If $\theta$ is not the maximal row-regular tableau, then $i < m - 2(n-1)$ and $\ell < n +m$.
\end{proof}

\begin{lemma}
\label{lem:ext}
Let $\mu = (m,n)$ be a triangular $2$-partition and $\theta$ an $i$-row-regular tableau on $\mu$. If $\lambda$ is the partition $(m+k,n)$ with $k > 0$ and $\theta'$ the tableau extending $\theta$ on $\lambda$ by labeling the extra cells with $m+n+1, m+n+2, \dots m+n+k$, then $\lambda$ is also a triangular partition and $\theta'$ is an $i$-row-regular tableau on $\lambda$.

We say that $\lambda$ is the extension of $\mu$ by $k$ cells.
\end{lemma}

\begin{proof}
This is direct by definition. As $\mu$ is triangular, we have $n \leq \lceil \frac{m}{2} \rceil \leq \lceil \frac{m+k}{2} \rceil$ so $\lambda$ is triangular. The labels of the upper row are not changed and the condition on $i$ is still satisfied as we have $1 \leq i \leq m - 2(n-1) \leq m + k - 2(n-1)$.
\end{proof}

\begin{proposition}
\label{prop:triang-tableau}
Let $\lambda = (m,n)$ be a triangular partition, then the triangular Young tableau of $\lambda$ is a row-regular tableau:
\begin{itemize}
\item if $n = \lceil \frac{m}{2} \rceil$, it is the maximal row-regular tableau;
\item if $n < \lceil \frac{m}{2} \rceil$, it is the pre-maximal row-regular tableau.
\end{itemize}
\end{proposition}

In other words, the maximal value of the tableau is placed on either one of the two rows depending on the length $n$, then the values alternate between the two rows until the shortest row is filled. For the partition $(9,4)$ of Figure~\ref{fig:row-regular}, as $4 < \lceil \frac{9}{2} \rceil$, the triangular tableau is the pre-maximal row-regular tableau. We show some other examples on Figure~\ref{fig:Alt}.

\begin{figure}[ht]
    \centering
    \input{figures/triangular_2tableaux.tex}

    \caption{The triangular tableaux of $(7,2)$, $(3,2)$ and $(4,2)$.}
    \label{fig:Alt}
\end{figure}

\begin{figure}[ht]
\input{figures/example-triangular-max-row-regular}
\caption{Examples illustrating the proof of Proposition~\ref{prop:triang-tableau}}
\label{fig:example-triangular-max}
\end{figure}

Row-regular tableaux are somehow a generalization of the triangular tableau. We prove later that they are sim-sym tableaux and, in most cases, the only possible sim-sym tableaux on triangular $2$-partitions. First we show a useful property generalizing in a sense Definition~\ref{def:tyt}: one can ``remove'' the cells of a row-regular tableau one by one by keeping the sub-partition triangular. 

\begin{proposition}
\label{prop:row-reg-reduced}
Let $\lambda = (m,n)$ be a triangular $2$-partition and $\theta$ the  $i$-row-regular tableau on~$\lambda$. Then for $n+i \leq k \leq m+n$, the sub-partition $\mu$ formed by the cells of $\lambda$ with labels smaller than or equal to $k$ is a triangular $2$-partition and the restriction of $\theta$ to $\mu$ is a row-regular tableau of $\mu$.
\end{proposition}

\begin{proof}
Let us first prove the case where $k \in \lbrace n+i, n+i+2, \dots, n+i+2(n-1) \rbrace$. In other words, $k = n +i +2 \ell$ with $\ell \leq n-1$. Then $\mu$ is equal to $(m',n')$ with $m' = n+i-1+\ell$ and $n' = \ell +1$. We have $m'-n'+1 = n+i-1 \geq n$ and $n' \leq n$ which proves that $\mu$ is triangular. Now the labels on the upper row of $\mu$ are given by $n+i, n+i+2, \dots, n+i+2 \ell$ which can be rewritten as $n' + j, n' + j +2, \dots, n' + j + 2 (n' -1)$ with $j = n - n' +i$. As $i \geq 1$ and $n \geq n'$, we have $j \geq 1$. Besides, we have $m' - 2(n'-1) = n + i - 1 - \ell = n - n' +i = j$. So the restriction of $\theta$ to $\mu$ is the $j$-row-regular tableau on $\mu$.

The other cases can be obtained by Lemma~\ref{lem:ext}. Indeed, if $k = n+i+2 \ell + 1$ with $\ell \leq n-2$, we have $\mu = (n+i+\ell, \ell +1)$ which is the extension of $(n+i+\ell-1, \ell+1)$ by one cell, which was proven before. If $k \geq n+i+2(n-1)$, then we have $\mu = (k-n,n)$ which is an extension of $(i + 2(n-1), n)$, which was proven before (case $\ell = n-1$).
\end{proof}

\subsection{Explicit $(q,t)$-enumeration and Schur positivity}
\label{subsec:qt-enum}

The goal of this section is to prove the following theorem on row-regular tableaux.

\begin{theorem}
\label{thm:row-regular-athqt}
For $\lambda = (m,n)$ a triangular partition and $\theta$ a row-regular tableau, then

\begin{equation}
\label{eq:alqt2-theta}
\Athqt = \sum_{0 \leq d \leq \min(n, m-n)} s_{m+n-2d,d}(q,t).
\end{equation}
\end{theorem}

\begin{remark}
\label{rem:ald}
Note that for all $0 \leq d \leq n$, then $m+n-2d \geq m-n$. So for all $0 \leq d \leq \min(n,m-n)$, we have that $(m+n-2d,d)$ is indeed a partition.
\end{remark}

This proves in particular that all row-regular tableaux are sim-sym tableaux. Besides, as we have proved in Proposition~\ref{prop:triang-tableau} that the triangular tableau is a row-regular tableau, we obtain the following.

\begin{corollary}
\label{thm:alqt2}
Conjecture~\ref{conj:schur} is satisfied on $2$-partitions and for $\lambda = (m,n)$, then 
\begin{equation}
\label{eq:alqt2}
\Alqt = \sum_{0 \leq d \leq \min(n, m-n)} s_{m+n-2d,d}(q,t).
\end{equation}
\end{corollary}

This last result was already stated in~\cite{Berg1} and the formula was actually obtained by~\cite[Theorem 3.1]{GHSR20} which explores a similar $(q,t)$-enumeration in a different combinatorial context. 

Note that because $\lambda$ is a triangular partition, we always have $n \leq \lceil \frac{m}{2} \rceil$, so $\min(n, m-n) = n$ in all cases except when $m$ is odd and $n = \frac{m+1}{2}$. For example, on the partition $\lambda = (9,4)$, we get

\begin{equation}
\Alqt = s_{13} + s_{11,1} + s_{9,2} + s_{7,3} + s_{5,4}.
\end{equation}

Before we prove the theorem, let us notice a few facts. Using the definition of Schur functions given in~\eqref{eq:schur}, remember that 

\begin{equation}
s_{b,a} = q^b t^a + q^{b-1} t^{a+1} + \dots + q^a t^b
\end{equation}
with $a \leq b$. Indeed, these are all the ways to fill a $2$-partitions with values $1$ and $2$ as a semi-standard Young tableau: the first $a$ cells of the lower row are filled with $1$, the second row is filled with $2$ and the $b-a$ remaining cells of the lower row are filled with a series of~$1$ (possibly empty) followed by a series of $2$ (possibly empty). In particular, the degree is always~$a+b$. 

In particular, all summands of $s_{b,a}$ are distinct. Besides, the Schur functions summed in~\eqref{eq:alqt2-theta} are all of different degrees, so all $q^it^j$ summands of~\eqref{eq:alqt2-theta} are disctincts. In fact, each Schur function in the sum corresponds to a given deficit $d$ and sums the elements $q^{\arealm} t^{\simthm}$ for all triangular Dyck paths of deficit $d$.

It follows that the theorem is the direct consequence of the following result.

\begin{proposition}
\label{prop:row-regular-exact-aread}
Let $\lambda = (m,n)$ be a triangular partition and $\theta$ a row-regular tableau on $\lambda$. Then, for each $d$ such that $0 \leq d \leq \min(n,m-n)$, and each $a$, such that $d \leq a \leq m+n-2d$, there is a unique triangular Dyck path $(\lambda, \mu)$ with $\defthm = d$ and $\arealm = a$. These are the only possible triangular Dyck paths of $\lambda$.
\end{proposition}

We give a constructive proof. Remember that a triangular Dyck path is given by a subparition $\mu$ of $\lambda$. It can be constructed by ``removing'' cells from $\lambda$.

\begin{definition}
Let $\lambda = (m,n)$ be a triangular partition, $\theta$ the $i$-row-regular tableau on $\lambda$, then for $d$ and $a$ such that $0 \leq d \leq \min(n,m-n)$ and $d \leq a \leq m+n-2d$, we define $\imdal$ the triangular Dyck paths constructed as follows:
\begin{itemize}
\item If $a = d$, then 
\begin{itemize}
\item Case 1: if the maximal label $m+n$ is on the lower row, we remove $d$ cells from the upper row.
\item Case 2: otherwise (the maximal label $m+n$ is on the upper row), we remove $d$ cells from the lower row.
\end{itemize}
\item If $d < a \leq m - d - i +2$, we first remove the $a-d$ cells from $\lambda$ with highest labels, then
\begin{itemize}
\item Case 1: if the remaining cell with maximal label ($m+n-(a-d)$) is on the lower row, we remove $d$ cells from the upper row.
\item Case 2: otherwise (the remaining maximal label $m+n-(a-d)$ is on the upper row), we remove $d$ cells from the lower row.
\end{itemize}
\item If $m - d - i + 2 < a \leq m+n -2d$, the triangular Dyck path is given by the subpartition $(m+n-a-d, d)$.
\end{itemize}
\end{definition}

We illustrate this construction on an example in Figure~\ref{fig:2part-aread-construction} with the three row-regular tableaux of the partition $(9,4)$ and all possible values of $a$ for $d=2$. The first line corresponds to the $a=d$ case. For the first two tableaux, the maximal label cell $13$ is on the lower row, so the subparition has been obtained by removing two cells from the upper row. On the last tableau (which is the maximal row-regular tableau), the cell with label $13$ is on the upper row, so we have removed two cells from the lower row.

Below the first line, we see an illustration of the $d < a \leq m - d - i +2$ cases. We have written a temporary partition $\lambda'$ corresponding to the partition $\lambda$ where the $a-d$ maximal cells have been removed. Then $2$ extra cells are removed from $\lambda'$ either on the lower or upper row, depending on the position of cell with label $m+n-(a-d)$. The last rows illustrate the $a > m - d - i + 2$ cases. 

\begin{figure}[ht]
\input{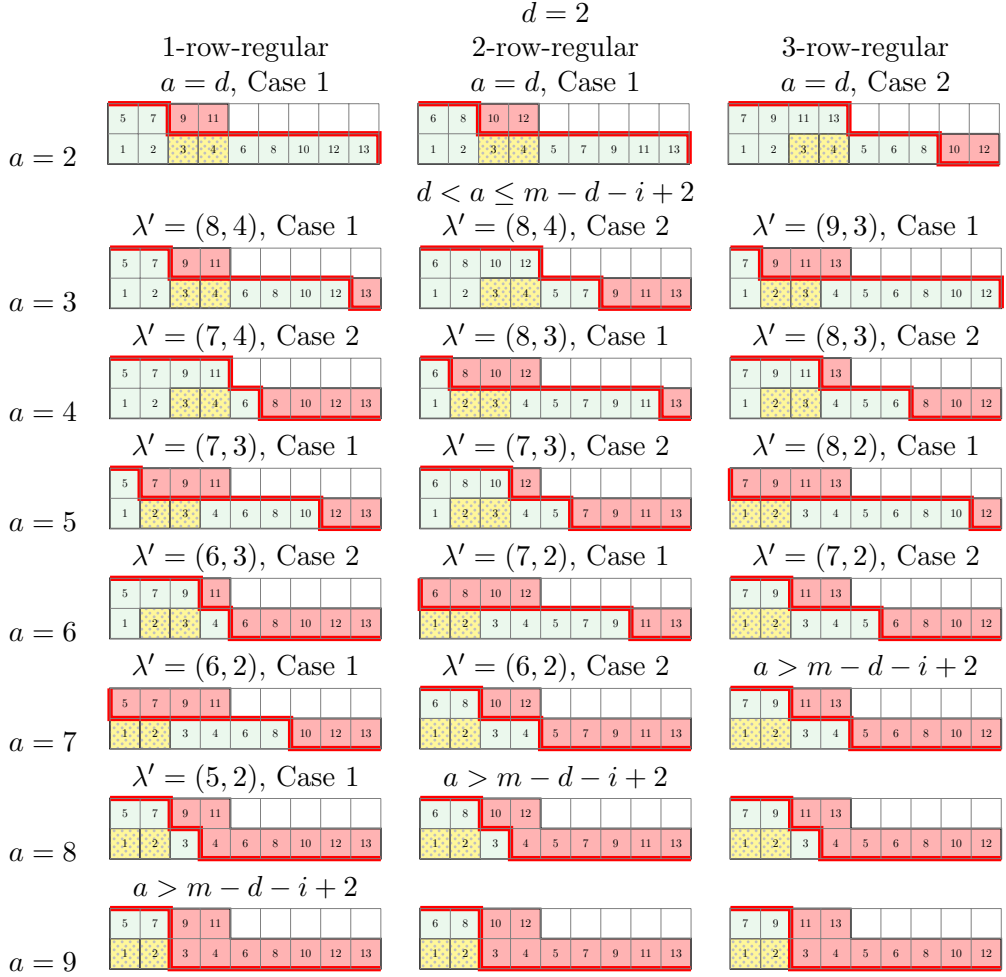}
\caption{Illustration of the proof of Lemma~\ref{lem:2part-aread-construction} on the three row-regular tableaux of $\lambda = (9,4)$.}
\label{fig:2part-aread-construction}
\end{figure}

\begin{lemma}
\label{lem:2part-aread-construction}
Let $\lambda = (m,n)$ be a triangular partition and $\theta$ the $i$-row-regular tableau on $\lambda$. Then $\imdal = (\lambda, \mu)$ is well defined and we have $\defthm = d$ and $\arealm = a$.
\end{lemma}

\begin{proof}
We show the result on the different cases depending on the value of $a$.

\paragraph{Case $a=d$} First note that for $0 \leq d \leq \min(n,m-n)$, it is always possible to construct the subpartition $\mu$ by removing $d$ cells from upper / lower row as we have both $d \leq n$ and $d \leq m-n$. Now notice that we have removed the cells from the row which did not contain the maximal value. We write $D := \lbrace n-d+1, n-d+2, \dots, n \rbrace$ and call the cells labeled by $D$, the $D$ cells. In all cases, we obtain that the $D$ cells are deficit cells.

Indeed, in Case 1, the removed cells are right above the $D$ cells and are labeled with $n+i +2(n-d), n+i + 2(n-d)+2, \dots, n+i+2(n-1)$ with $i \geq 1$, whereas the cells labeled $n+i +2(n-d) + 1, n+i + 2(n-d)+3, \dots, n+i+2(n-1) + 1$ lie to the right of the $D$ cells which create the deficit.

In Case 2, the removed cells are labeled $m+n-1-2(d-1), \dots, m+n-3, m+n-1$ by definition of the maximal row-regular tableau. Besides, the cells right above the $D$ cells are labeled $m+n-2(d-1), \dots, m+n-2, m+n$ which create the deficit.

\paragraph{Case $d < a \leq m - d - i +2$} We use Proposition~\ref{prop:row-reg-reduced} with $k = m+ n - (d-a)$. In other other words, we look at the difference between $a$ and $d$ and remove this number of cells from $\lambda$ (in reverse order of their label in $\theta$). We obtain a partition $\lambda' = (m', n')$ which by Proposition~\ref{prop:row-reg-reduced} is still triangular and labeled by a row-regular tableau. Now if we can construct a triangular Dyck path $(\lambda', \mu)$ with both deficit and area equal to $d$, this gives a triangular Dyck path $(\lambda, \mu)$ with deficit $d$ and area $a$. This construction is only possible if $d \leq \min(n', m'-n')$ and can be done using the $a = d$ case treated above on $\lambda'$. It is illustrated on Figure~\ref{fig:2part-aread-construction}. See for example, the $a = 4$ example on the $1$-row-regular tableau. As $a - d = 2$, we remove the two cells labeled $13$ and $12$ from $\lambda$ and obtain $\lambda' = (7,4)$. Now the maximal label is $11$ and placed on the upper row, so we are in Case 2 and we remove two extra cells from the lower row. We obtain $\mu = (5,4)$. The deficit of both $(\lambda', \mu)$ and $(\lambda, \mu)$ is $2$ and the area of $(\lambda, \mu)$ is $2+2 = 4$. We illustrate many other cases on Figure~\ref{fig:2part-aread-construction}.

We need to prove that $a \leq m+n-2d$ implies that the condition $d \leq \min(n',m'-n')$ is satisfied. As $\theta$ is an $i$-row-regular tableau, the label of the first cell of the upper row is $n+i$ and the label of the $d^{th}$ cell is $n+i+2(d-1)$. Now $a \leq m - d - i +2$ implies that $n+i+2(d-1) \leq m+n - (a -d)$. In other words, the $d^{th}$ cell of the upper row still belongs to $\lambda'$ which translates as $d \leq n'$ and proves that the construction is possible.

\paragraph{Case $m - d - i + 2 < a \leq m+n -2d$} We take $\mu = (m+n-a-d, d)$ and claim that the triangular Dyck path $(\lambda, \mu)$ has the right properties. First note that $\mu$ is well defined because $a \leq m +n -2d$ by definition, which gives $m+n-a-d \leq d$. The area of $(\lambda, \mu)$ is given by $n - d + m - (m+n-a-d) = a$. We need to prove that the deficit is~$d$. The $d$ cells of the upper line of $\mu$ have labels $n+i, n+i+2, \dots, n+i + 2(d-1)$. Beside, the first $n+i-1$ cells of the lower row of $\lambda$ are labeled by $1,2,\dots, n+i-1$ by definition of the row-regular tableau. We want to show that the first $d$ cells of the lower row of $(\lambda, \mu)$ are deficit cells. We look at the cell with label $n+i-1$. As we have $a > m-n-i+2$, this gives that $n+i-1 > m+n-d-a$, so the cell with label $n+i-1$ is not contained in $\mu$. This implies that the following cells of the lower line $n+i+2-1, \dots, n+i+2(d-1)-1$ are also not contained in $\mu$, whereas the first $d$ cells of the upper row $n+i, n+i+2, \dots, n+i + 2(d-1)$ are: this proves that $(\lambda, \mu)$ has deficit $d$. This is illustrated on Figure~\ref{fig:2part-aread-construction}: on the $1$-row-regular tableau, it corresponds to the $a=9$ case, on the $2$-row-regular tableau, to the $a=8$ and $a=9$ cases, and on the $3$-row-regular tableau, to the $a=7$, $a=8$, and $a=9$ cases.
\end{proof}

\begin{proof}[Proof of Proposition~\ref{prop:row-regular-exact-aread}]
Let $\lambda = (m,n)$ be a triangular partition  and $\theta$ a row-regular tableau on $\lambda$.
By Lemma~\ref{lem:2part-aread-construction}, we know that for each $d$ such that $0 \leq d \leq \min(n,m-n)$, and each $a$, such that $d \leq a \leq m+n-2d$, the triangular Dyck path $\imdal$ has deficit $d$ and area~$a$. This implies in particular that for $(d',a') \neq (d,a)$, then $\imp{d'}{a'}{\lambda} \neq \imdal$, the application is injective. We can prove that it is also surjective using a simple counting argument.
Let $D$ be the number of triangular Dyck paths obtained by the $\imdal$ construction and $S$ the number of subpartitions of $\lambda$. If we prove that $D = S$, we are done.

The value $S$, can be easily computed. Let $\mu = (m',n')$ be a subpartition of $\lambda$. Then we have $0 \leq n' \leq n$ and $n' \leq m' \leq m$ which gives $m - n' + 1$ possible subpartitions for each value~$n'$. We obtain

\begin{equation}
S = \sum_{n' = 0}^{n} m - n' + 1 = \frac{(2m-n+2)(n+1)}{2}.
\end{equation}

Now, for each $d$ with $0 \leq d \leq \min(n,m-n)$, we can construct $m+n-2d - d + 1 = m+n-3d+1$ triangular Dyck paths and so we obtain
\begin{equation}
D = \sum_{d=0}^{\min(m,m-n)} m+n-3d+1.
\end{equation}

If $\min(m-n,n) = n$, then we get
\begin{equation}
D = \frac{(m+n+1+m+n-3n+1)(n+1)}{2} = S.
\end{equation}

If $\min(m-n,n) = m-n$, then we have
\begin{equation}
D = \frac{(-m+5n+2)(m-n+1)}{2}.
\end{equation}
As we are working with triangular partitions, the only case where $\min(m-n,n) = m-n$ is when $m$ is odd and $n =\frac{m+1}{2}$. By simple substitution in $D$ and $A$, we obtain that numerators in both are given by $\frac{3(m+1)(m+3)}{4}$ and so $D = S$ in all cases.
\end{proof}

\begin{proof}[Proof of Theorem~\ref{thm:row-regular-athqt}]
Let $\lambda = (m,n)$ be a triangular partition and $\theta$ a row-regular tableau on $\lambda$. Let $d$ be such that $0 \leq d \leq \min(n,m-n)$. Then the summand $s_{m+n-2d, d}(q,t)$ is equal to 

\begin{equation}
 \sum_{\mu} q^{\arealm} t^{\simthm}
\end{equation}
summed over the subpartition $\mu$ of $\lambda$ such that $\defthm = d$. Indeed, we have

\begin{equation}
s_{m+n-2d,d}(q,t) = \sum_{a=d}^{m+n-2d} q^a t^{m+n-d-a}.
\end{equation}
By Proposition~\ref{prop:row-regular-exact-aread}, we know that for a fixed $d$, there is exactly one triangular Dyck path $\imdal = (\lambda, \mu)$ for each value $a$ such that $d \leq a \leq m+n-2d$. Each of these triangular Dyck paths corresponds to a summand
\begin{equation}
q^{\arealm} t^{\simthm} = q^{\arealm} t^{m + n - a - \defthm} = q^a t^{m+n-a-d}.
\end{equation}

As these are the only possible triangular Dyck paths and $d$ can only vary between $0$ and $\min(n,m-n)$, we obtain the formula.
\end{proof}

\subsection{Sim-sym tableaux on 2-partitions}

We have seen that row-regular tableaux are sim-sym. In most cases, this is actually an equivalence relation.

\begin{theorem}[Characterization of sim-sym tableaux on $2$-partitions]
    \label{thm: simsym-row-reg}
Let $\lambda = (m,n)$ be a triangular partition such that $n \neq 2$, then $\theta$ is a sim-sym tableau if and only if it is a row-regular tableau. 

If $n=2$, then $\theta$ is a sim-sym tableau if and only if $\theta$ is a row-regular tableau or the values on the upper line are $2$ and $5$. 

\end{theorem}

\begin{figure}[ht]
    \centering
    \input{figures/simsym6-2-all}
    \caption{The 5 different sim-sym tableaux on (6,2).}
    \label{symT}
\end{figure}

\begin{definition}
    For  $\lambda$ a partition, $\theta$ a sim-sym tableau on $\lambda$ and $\mu \subset \lambda$, 
    we write $\atm(q,t)= q^{\arealm}t^{\simthm} $
\end{definition}

For the next two lemmas, we consider $\lambda = (m,n)$ a triangular partition and $\theta$ a sim-sym tableau.

\begin{lemma}
    \label{lem:area-sim-def}
    Let $\mu \subset \lambda$, then $\arealm \geq \defthm$ and $\simthm \geq \defthm$
\end{lemma}

\begin{proof}
Using the formula of~\ref{thm:alqt2}, as $\theta$ is sim-sym, we have
    \begin{align}
        \sum_{\mu}\atm(q,t)  &= \Alqt\\
        &= \sum_{0 \leq d \leq \min(n, m-n)} s_{m+n-2d,d}(q,t)
    \end{align}
    Hence, there exist $0 \leq d \leq \min(n,m-n)$ and $d \leq a \leq m+n - 2d$ such that:
    $$ \atm(q,t)  = q^at^{m+n-d-a}$$
    As such, $\arealm = a \geq d$ and $\simthm = m+n-d-a \geq m+n -d - (m+n-2d)=d$

\end{proof}

\begin{lemma}
    \label{lem:equal-subpart}
    Let $\mu,\tau$ two subpartitions of $\lambda$ such that $\arealm = \arealt$ and $\simthm = \simtht$.
    Then $\mu = \tau$
\end{lemma}

\begin{proof}
    Let us consider such $\mu$ and $\tau$. We have $\atm(q,t) = a_{\theta}(\tau)(q,t)$.\\
    However, each monomial in a Schur function $s_{i,j}(q,t)$ only appears with multiplicity 1.
    Furthermore, for any $d \neq d'$, $deg(s_{m+n-2d,d}) = m+n-d \neq deg(s_{m+n-2d',d'}) $.
    As such, no monomial of $s_{m+n-2d,d}$ appears in $s_{m+n-2d',d'}$.
    Hence, $\atm(q,t)$ appears with multiplicity 1 in $\Athqt$ and $\mu = \tau$.
\end{proof}

\begin{lemma}
    \label{lem:def1}
    For any $1\leq l \leq |\lambda|$, there is a subpartition $\mu$ of size $|\mu|=l$, such that $\defthm=1$.
\end{lemma}

\begin{proof}
    This is a direct corollary of Corollary~\ref{thm:alqt2} and Lemma~\ref{lem:2part-aread-construction}.
\end{proof}

\begin{lemma}
    \label{lem:special-case-n-2}
    If n = 2, and $\theta$ is the standard young tableau with values 2 and 5 on the upper line, then $\theta$ is sim-sym.
\end{lemma}

\begin{proof}
    By Theorem \ref{thm:row-regular-athqt}, we know that the tableay $\theta'$ with values 3 and 5 on the upper line is sim-sym.
    The only difference between $\theta$ and $\theta'$ is swapping the values 2 and 3.
    Swapping the value only changes the similarity of the first cell in the first line.
    
    For any $\mu \subset \lambda$, if $(2,1) \subset \mu$ then the first cell is a similar cell in both cases.
    Otherwise, if $ (2,1) \not\subset \mu$ but $(3) \subset \mu$ then the first cell is a deficit cell in both cases.
    Finally, if $\mu = (1)$, the first cell is also a similar cell in both cases, and if $\mu = \emptyset$, it's in the area.
    
    As such, the only subpartition in which the swapping has an impact are $(2)$ and $(1,1)$.
    We can easily see that $\simp{\theta}{(2)} = \simp{\theta'}{(1,1)}$ and $\simp{\theta'}{(2)} = \simp{\theta}{(1,1)}$.
    As such $\Athqt = A_{\theta'}(q,t)$ and we get that $\theta$ is sim-sym.
\end{proof}

\begin{lemma}
    \label{lem:sim-sym-to-row-reg}
    For $\theta$ a tableau on $\lambda$, if $\theta$ is sim-sym, then $\theta$ is row-regular or n=2 and the values on the upper line are 2 and 5.
\end{lemma}

\begin{proof}
	Let us suppose that $\theta$ is a standard Young tableau on $\lambda$ and is not row-regular. We show that 
	the only possibility is that $n=2$ and the upper values are 2 and 5. All other cases lead to a contraction.

    Indeed, there are 3 possible cases:
    \begin{enumerate}
    \item The biggest value on the upper is smaller or equal to $3n-2$;
    \item there exists two adjacent cells on the lower line with values $k$ and $k+i$ with $ i \geq 3$;
    \item there exists two adjacent cells on the upper line with values $k$ and $k+i$ with $ i \geq 3$.
    \end{enumerate}

    \paragraph{Case 1} As $\lambda$ is triangular, we have $m \geq 2n-1$. As such, we can consider $\mu = (2n-1)$.
    As $\theta$ is a standard Young tableau, the value in the n-th cell on the upper line is greater than the values of each cells to its left.
    As such, there are n cells on the upper row with a value lower or equal to $3n-2$
    Hence there can be only $3n-2 -n = 2n-2$ cells on the lower row with a value lower or equal to $3n-2$.
    Hence the value on the $(2n-2)$-th cell is necessarely $3n-1$ in this case.
    Therefore, the first n cells in $\mu$ are deficit cells (and are the only deficit cells).
    So $\defthm=n$ and $\simthm = 2n-1 - \defthm = n-1$. This is in contradiction with Lemma~\ref{lem:area-sim-def}.

    \paragraph{Case 2}In this case, there are adjacent cells in the upper row with value $k+1$ to $k+i-1$.
    \begin{itemize}

        \item If the cell numbered $k+i-1$ is above $k$, we can consider the following subpartition:\\
    $\mu$ whose last cell in the lower line is valued $k+i$, and whith no cell in the upper line. 
    $\mu$ is of size at least 3, because there are at least $i \geq 3$ cells above $\mu$ (the cells numbered $k+1$ to $k+i-1$ and the one above $k+i$).
    However, $\defthm= |\mu|-1 \geq 2 \geq \simthm = 1$, which is in contradiction with Lemma~\ref{lem:area-sim-def}.

    \item Otherwise, the cell above $k$ is greater than $k+1$. 
    We can consider the 3 following subpartition of size $k+1$:
    \begin{itemize}
    \item $\mu_1$ whose last cell in the lower line is valued $k+i$, and whose last cell on the upper line is stopping short of the cell valued $k+1$.
    \item $\mu_2$ whose last cell in the lower line is valued $k$, and whose last cell on the upper line is valued $k+1$.
    \item And finally, $\mu_3$,whose last cell in the lower line is juste before the cell valued $k$, and whose last cell on the upper line is valued $k+2$.
    \end{itemize}
    We get $\defp{\theta}{\mu_2} = 0$, $\defp{\theta}{\mu_1} = i-1 \geq 2$ and $\defp{\theta}{\mu_3} = 2$.
    Furthermore, we can see than every other subparitions of size $k+1$ has a deficit greater than $i-1$ if it has more cell on the upper line than $\mu_1$ or greater than 2 otherwise.
    As such, there is no subpartition of size $k+1$ and of def 1, which is in direct contradiction with Lemma~\ref{lem:def1}.
    \end{itemize}

    \paragraph{Case 3}In this case, there are adjacent cells in the lower line with value $k+1$, $k+2$,\ldots, $k+i-1$.
    We suppose we are not in any of the previous cases.
    There are 3 possibilities:
    \begin{itemize}
        \item if the cell  numbered $k+i$ is not above the cell $k+1$.
        We can consider the 3 following subpartitions of size $k+1$:
        \begin{itemize}
        \item $\mu_1$ whose last cell in the lower line is valued $k+2$, and whose last cell on the upper line is stopping short of the cell valued $k$.
        \item $\mu_2$ whose last cell in the lower line is valued $k+1$, and whose last cell on the upper line is valued $k$.
        \item And finally, $\mu_3$, whose last cell in the lower line is juste before the cell valued $k+1$, and whose last cell on the upper line is valued $k+i$.
        \end{itemize}
        
    We get $\defp{\theta}{\mu_2} = 0$, the cell in $\mu_2$ are exactly the ones numbered 1 to $k+1$.
    And $\defp{\theta}{\mu_1} = 1 = \defp{\theta}{\mu_3}$.
    In the first subpartition, the only deficit cell is the cell under~$k$, and in the third, the only deficit cell is the one under $k+i$.
    This is in contradiction with Lemma~\ref{lem:def1}.
    \item if the cell numbered $k+i$ is above $k+1$ and $i\geq 4$.
    We use the same idea, but with 3 subpartitions of size $k+2$ that each get one more cell in the lower row than their respective previous subpartition.
    Hence, we also get a contradiction with Lemma~\ref{lem:area-sim-def}.
    \item Finally, if the cell numbered $k+i$ is above $k+1$ and $i = 3$.
    We can consider that we are in the first instance of adjacent cells with consecutive numbers on the lower line, and that the other instance are of this nature.
    Because there is not case 2 in $\theta$ and we are considering the first instance of case 3, we can deduce that the cells before $k$ are numbered $k-2$, $k-4$, \ldots until the first cells of the upper line.
    Likewise, the cells before $k+1$ are numbered $k-1$, $k-3$, \ldots until the second cell of the lower line (whose numbered is one greater than the number of the first cell in the upper line).
    As such, the first and second cells in the lower line are numbered 1 and 3 and the first in the second line is numbered 2.
    Because there are 2 subpartitions of size 3, one of the has to have deficit 0 and the other 1.
    $(2,1)$ has deficit 0, hence $(3)$ must have deficit $1$.
    Because 1 is a deficit cell, 3 should not be a deficit cell, hence the number in the third cell in the lower line is smaller than the number of the second cell in the upper line.
    As such, we get that the third cell is 4 and the second on the upper line is 5.
    If there are any cells after 5, they would have to be numbered 7, then 9,\ldots otherwise we would have the first subcase.
    But that we give a contradiction with the case 1, hence there are only 2 cells on the upper line, and they are numbered 2 and 5, which gives us the special case.
    \end{itemize}

\end{proof}

\begin{proof}[Proof of Theorem \ref{thm: simsym-row-reg}]
    We just combine Lemma ~ref{lem:sim-sym-to-row-reg} and Theorem~\ref{thm:row-regular-athqt}.
\end{proof}

\section{Lattice and interval enumeration}
\label{sec:lattice}

\subsection{Tamari lattice and motivation}

The Tamari lattice~\cite{Tam51} is a well known lattice on Catalan objects, often defined on Dyck paths. It has many known generalizations such as the $m$-Tamari lattices~\cite{BPR12} and the $\nu$-Tamari lattices~\cite{PRV15}. The $m$-Tamari lattices in particular have been introduced in the context of trivariate diagonal harmonic polynomials and are the subject of many interesting conjectures~\cite{PR12}.

More precisely, the  polynomials $\Alqt$ of~\eqref{eq:alqt} sometimes occur as the expansion in two variables of the characters of a certain $GL_2$-action. Hence, they are Schur positive, as the Schur functions are the characters of the irreducibles in this context. For all triangular partitions~$\lambda$, one can define a symmetric function $A_\lambda$ corresponding to the Hilbert series of alternating component of multivariate diagonal harmonic polynomials~\cite{Ber22}. For example, we have that 

\begin{align}
A_{2,1} &= s_{3} + s_{1,1}, \\
A_{3,2,1} &= s_{6} + s_{4,1} + s_{3,1} + s_{1,1,1}.
\end{align}

When expanded on two variables, the term $s_{1,1,1}(q,t)$ in $A_{3,2,1}$ is $0$ and we get \eqref{eq:qtschur_a321}, the $(q,t)$-enumeration of Dyck paths. If we expand the functions on three variables, we obtain $A_{2,1}(1,1,1) = 13$ and $A_{3,2,1}(1,1,1) = 68$ the number of intervals in the Tamari lattice of sizes $3$ and $4$ respectively. This is part of a larger conjecture on $A_\lambda$ and the Tamari lattice related to trivariate diagonal harmonic polynomials. For $\lambda = (n-1, n-2, \dots, 1)$ the staircase partition, it is conjectured that $A_\lambda(q,t,r)$ enumerates the intervals of the Tamari lattice of size $n$ with $q$ counting the \emph{distance} between the two elements of the interval and $t$ the \emph{dinv} of the maximal element. For example, we have

\begin{equation}
A_{3,2,1}(q,1,1) = q^{6} + 2 q^{5} + 5 q^{4} + 10 q^{3} + 15 q^{2} + 21 q + 14.
\end{equation}
The distance of an interval $[d_1, d_2]$ is the length of the longest chain in the lattice between $d_1$ and $d_2$. The constant term $14$ correspond to the $14$ elements with ``distance 0'', \emph{i.e.}, the 14 intervals $[d,d]$ where $d$ is a Dyck path of size $4$. The term $q^6$ corresponds to the interval $[\hat{0},\hat{1}]$ between the minimal and maximal elements in the lattice. The full expansion on $A_{2,1}$ gives

\begin{equation}
A_{2,1}(q,t,r) = q^{3} + q^{2} t + q t^{2} + t^{3} + q^{2} r + q t r + t^{2} r + q r^{2} + t r^{2} + r^{3} + q t + q r + t r.
\end{equation}

The combinatorial interpretation of the third parameter $r$ is a difficult open question listed as Problem 5 in~\cite{PR12}.

\subsection{Conjecture on the $\nu$-Tamari lattice}

The aforementioned conjectures and combinatorial interpretation naturally extend to $m$-Tamari lattices, where $\lambda = (m (n-1), m (n-2), \dots, m, 0)$ is an $m$-staircase partition. Dyck paths and $m$-Dyck paths can be understood as subpartitions of the staircase and $m$-staircase partitions  respectively. It is then a natural question to explore possible combinatorial interpretations of the $A_\lambda$ polynomials expanded on three variables for $\lambda$ a triangular partition. This interpretation should take the form of an interval enumeration on a lattice structure involving triangular Dyck paths.

In~\cite{PRV15}, Préville-Ratelle and Viennot introduce a generalization of the Tamari lattice, now mostly referred to as the $\nu$-Tamari lattice and largely studied~\cite{CPS16,FPR17}. These lattices are defined on subparitions of \emph{any} given partition. In particular, it can be defined on triangular Dyck paths of a given triangular partition $\lambda$. Below is the definition we use in this paper.

\begin{definition}
Let $\lambda$ be a triangular partition and $(\mu, \lambda)$. We write $\lambda = (\lambda_1, \lambda_2, \dots)$ and $\mu = (\mu_1, \mu_2, \dots)$. We choose $j$ such that $\mu_j > \mu_{j+1}$ (the end of line $j$ is a \emph{corner}) and $v := \lambda_j - \mu_j$. Let $i < j$ be the minimal possible value, such that for all $i < k < j$, $\lambda_k - \mu_k > v$. Then the \emph{rotation} of $\mu$ at line $j$ is the triangular Dyck path $(\alpha, \lambda)$ such that:
\begin{itemize}
\item $\alpha_k = \mu_k$ for all $k\leq i$ and all $k > j$;
\item $\alpha_k = \mu_k - 1$ for all $k$ such that $i < k \leq j$.
\end{itemize}

We say that $(\alpha, \lambda)$ \defn{covers} $(\mu, \lambda)$. We define the $\nu$-Tamari poset on subparitions of $\lambda$ as the partial order 
obtained by transitive closure of the rotation operation. In other words, $(\mu, \lambda) \prec (\tau, \lambda)$ if and only if $\tau$ can be obtained from $\mu$ by a series of rotations.
\end{definition}

We show an example on Figure~\ref{fig:nutamari-rotation}. We basically choose a starting line and remove one box per line going down as long as the number of red boxes on the line is bigger than on the starting line. This coincides (up to some rotation) with the original definition of the $\nu$-Tamari lattice given in~\cite{PRV15}, where $\nu$ is the path drawn by the partition $\lambda$. In particular, this is a lattice : any two elements $\mu$ and $\alpha$ admit a \defn{join} (or least upper bound) written $\mu \join \alpha$ and a \defn{meet} (or greatest lower bound) written $\mu \meet \alpha$. See an example of the complete lattice for $\lambda = (4,2,1)$ on Figure~\ref{fig:nutamari-example-421}.

\begin{figure}[ht]
\input{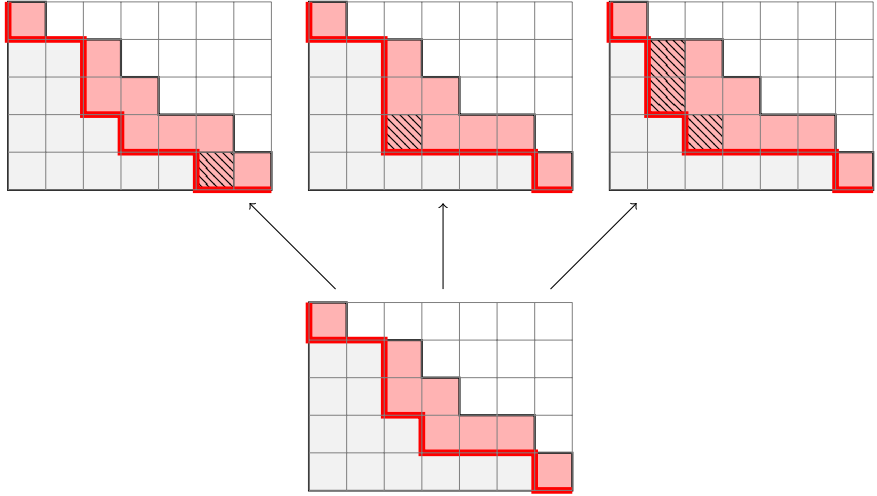}
\caption{Example of $\nu$-Tamari rotations on subpartitions}
\label{fig:nutamari-rotation}
\end{figure}

\begin{figure}[ht]
\scalebox{.3}{\input{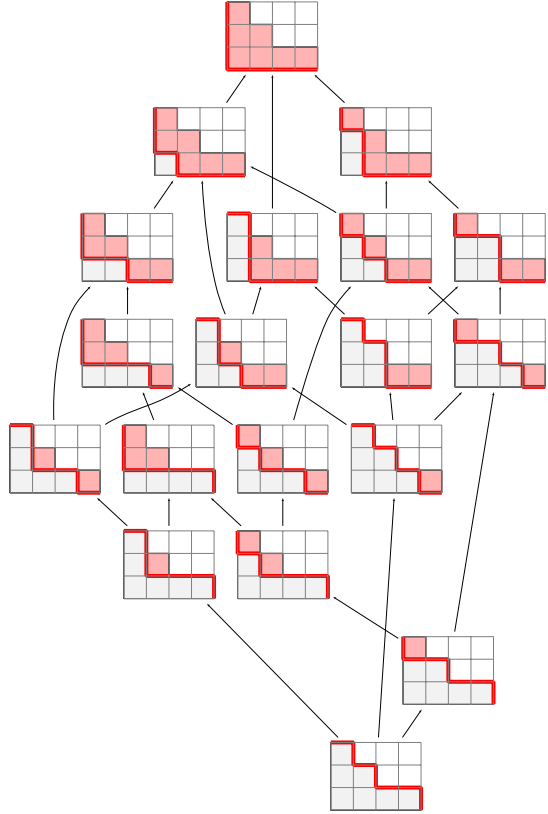}}
\caption{The $\nu$-Tamari lattice for $\lambda = (4,2,1)$}
\label{fig:nutamari-example-421}
\end{figure}

We write the following lemma which is a well known fact about the Tamari and $\nu$-Tamari lattices and will be useful later on. It only states that  the $\nu$-Tamari lattice is an extension of the (reversed) order on partition inclusion. In paritcular, the minimal element is $(\lambda, \lambda)$ and the maximal is $(\epsilon, \lambda)$ where $\epsilon$ is the empty partition.

\begin{lemma}
\label{lem:nutam-incl}
Let $\lambda$ be a triangular partition and $\mu$ and $\alpha$ two subpartitions such that $(\mu, \lambda) \preceq (\alpha, \lambda)$ in the $\nu$-Tamari lattice of $\lambda$. Then $\alpha \subseteq \mu$ as partitions, \emph{i.e.}, $\alpha_1 \leq \mu_1, \alpha_2 \leq \mu_2, \dots$.
\end{lemma}

\begin{proof}
Each cover relation removes at least one cell from the subpartition and never adds any.
\end{proof}

Remember that we are trying to find a combinatorial interpretation for some general $A_\lambda$ symmetric polynomials, for which we do not have a general formula. Anyway, we can summarize some ``good'' necessary properties in the problem below. 

\begin{definition}
Let $\lambda$ be a triangular partition and $\theta$ a sim-sym tableau on $\lambda$. Let $\preceq$ be a partial order on the triangular Dyck paths of $\lambda$. We define two statistics on intervals of $\preceq$. Let $\tau \preceq \mu$ be an interval, then
\begin{itemize}
\item $\dist(\tau,\mu)$ is the length of the longest chain between $\tau$ and $\mu$;
\item $\ssim_{\theta}(\tau,\mu) := \simthm$, \emph{i.e.}, the \emph{sim} of the interval is the \emph{sim} (relatively to the tableau~$\theta$) of the maximal element.
\end{itemize}
\end{definition}

\begin{problem}
\label{prob:lattice-problem}
Let $\lambda$ be a triangular partition and $\theta$ a sim-sym tableau on $\lambda$. Find a partial order $\preceq$ on the triangular Dyck paths of $\lambda$ such that the polynomial
\begin{equation}
\label{eq:dist-sim-pol}
\sum_{\tau \preceq \mu} q^{\dist(\tau, \mu)} t^{\ssim_\theta(\tau, \mu)}
\end{equation}
summed over the intervals of $\preceq$ is symmetric and Schur positive.
\end{problem}

Note that what we actually want is for the polynomial to be equal to some Schur positive function $A_\lambda$ expanded in $(q,t,1)$. Let us start with a negative result.

\begin{proposition}
The $\nu$ Tamari lattice does not provide a solution in general to Problem~\ref{prob:lattice-problem}.
\end{proposition}

\begin{proof}
Let us first give a counter-example where $\theta$ is fixed to be the triangular tableau. Let $\lambda = (4,2,1)$. The $\nu$-Tamari lattice is given as an example in Figure~\ref{fig:nutamari-example-421}. The triangular standard tableau is given on Figure~\ref{fig:triangular-421}.

\begin{figure}[ht]
\begin{center}
\input{figures/triangular-421}
\end{center}
\caption{The triangular tableau of $\lambda = (4,2,1)$}
\label{fig:triangular-421}
\end{figure}

If we compute the polynomial \eqref{eq:dist-sim-pol}, we obtain

\begin{align}
&q^{7} + q^{6} t + q^{5} t^{2} + q^{4} t^{3} + q^{3} t^{4} + q^{2} t^{5} + q t^{6} + t^{7} + q^{6} + 2 q^{5} t + 2 q^{4} t^{2} + 2 q^{3} t^{3} \\
&+ 2 q^{2} t^{4} + 3 q t^{5} + t^{6} + 2 q^{5} + 4 q^{4} t + 5 q^{3} t^{2} + 6 q^{2} t^{3} + 3 q t^{4} + 2 t^{5} + 3 q^{4} + 6 q^{3} t \nonumber \\
&+ 6 q^{2} t^{2} + 6 q t^{3} + 3 t^{4} + 4 q^{3} + 8 q^{2} t + 8 q t^{2} + 4 t^{3} + 4 q^{2} + 7 q t + 4 t^{2} + 3 q + 3 t + 1 \nonumber
\end{align}
which is not symmetric.

Nevertheless, it is possible to find a tableau for which the polynomial is symmetric (see the top-down tableau of Figure~\ref{fig:top-down}). One would like a counter example which fails on every tableau of $\lambda$. This is the case for $\lambda = (4,3,1)$. As the symmetry on intervals would imply the sim-area symmetry on triangular Dyck paths, one can test only sim-sym tableaux. There are four of them shown on Figure~\ref{fig:sim-sym-431}. The computed polynomials on each of them is not symmetric. All computations can be found in~\cite{PonSage25}.
\end{proof}

\begin{figure}[ht]
\input{figures/sim_sym_431}
\caption{Sim-sym tableaux of $\lambda = (4,3,1)$}
\label{fig:sim-sym-431}
\end{figure}

Even though the $\nu$-Tamari lattice does not give a general solution to problem~\ref{prob:lattice-problem}, it is still interesting to look at as it gives a solution in \emph{some} cases. The first observation is that the triangular tableau is not the natural choice for solving this problem. This is the reason we introduced the notion of sim-sym tableaux in the first place. Indeed, if $\lambda$ is partition of $n$, then a longest chain on the $\nu$-Tamari lattice is of length $n$ as it corresponds to removing one cell at a time. If we number these cells in reverse order of their removal, we obtain a standard tableau~$\theta$ and the intervals of this longest chain are actually enumerated by the Schur function~$s_n$. In general, the triangular tableau does not correspond to a longest chain in the $\nu$-Tamari lattice. See for example the triangular tableau of $(4,2,1)$ given on Figure~\ref{fig:triangular-421} and the $\nu$-Tamari lattice shown on Figure~\ref{fig:nutamari-example-421}.  After removing the cells labeled by $7$ and $6$, removing the cell labeled by $5$ forces us to remove the cell $4$ at the same time.

This is why we introduce the notion of top-down tableau.

\begin{definition}
Let $\lambda$ be a triangular partition of $n$. The top-down tableau of $\lambda$ is given by labeling with $n$ the last cell of the top line, then with $n-1$ the last cell of the line below and so on until we have labeled exactly one cell per line. We then redo the process on the remaining cells.
\end{definition}

Figure~\ref{fig:top-down} shows an example of this construction on the partition $(4,2,1)$.

\begin{figure}[ht]
\input{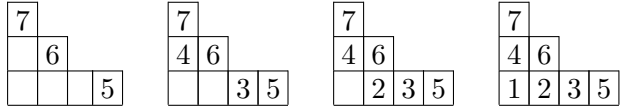}
\caption{Construction of the top-down tableau for $\lambda = (4,2,1)$.}
\label{fig:top-down}
\end{figure}

\begin{proposition}
The top-down tableau corresponds to a maximal chain in the $\nu$-Tamari lattice. In other words, it is possible to remove the cells one by one in the reverse order of the labels by following covering relations of the $\nu$-Tamari lattice.
\end{proposition}

\begin{proof}
Let $\lambda$ be a triangular partition of $n$, $k$ such that $1 \leq k \leq n$, and let $\mu$ be the subpartition formed by the cells with labels $c \leq k$. Let $\tilde{\mu}$ be the subpartition formed by the cells with labels $c < k$. We need to prove that $\tilde{\mu}$ can be obtained through a rotation on $\mu$. By definition of $\mu$ and of a standard tableau, the cell labeled by $k$ is at a corner and at the end of a certain line $i$ of $\mu$, so the rotation of $\mu$ at line $i$ is well defined and will remove at least this cell. The rotation also removes a cell from line $i-1$ (if it exists) if $\lambda_{i-1} - \mu_{i-1} > \lambda_{i} - \mu_{i}$. This is not possible by the top-down construction as every time a cell on line $i-1$ has been removed, a cell on line $i$ has been removed just before.
\end{proof}

Using the top-down tableau instead of the triangular tableau, for the partition $(4,2,1)$ we find that the $\nu$-Tamari lattice is a solution to Problem~\ref{prob:lattice-problem}. Indeed, we get that~\eqref{eq:dist-sim-pol} is symmetric and equal to the expansion of

\begin{equation}
s_{2,1,1} + s_{3,2} + s_{4,1} + s_{5,1} + s_{7}
\end{equation}
in $q,t,1$ (see~\cite{PonSage25} for the computation and test).

What is true for $(4,2,1)$ is true for many other partitions but not for all. Indeed, sometimes it cannot be true because the top-down tableau is not sim-sym. This is the case in particular for $(4,3,1)$, you can check that none of the sim-sym tableaux presented in Figure~\ref{fig:sim-sym-431} are the top-down tableau which is given on Figure~\ref{fig:top-down-431}. We see an example of triangular Dyck path with area equals to $3$ and sim equals to $2$. It can be checked that there is no triangular Dyck path with area equals to $2$ and sim equals to $3$.

\begin{figure}[ht]
\begin{center}
\scalebox{.5}{\input{figures/not-sim-sym-431}}

$\arealm = 3$, $\simthm = 2$
\end{center}
\caption{Counter example of area, sim symmetry on the top-down tableau of $\lambda = (4,3,1)$.}
\label{fig:top-down-431}
\end{figure}

Our exploration leads us to emit the following conjecture.

\begin{conjecture}
\label{conj:lattice}
If $\lambda$ is a triangular partition such that its top-down tableau $\theta$ is a sim-sym tableau, there exists a Schur-positive symmetric function $\lambda$ such that 
\begin{equation}
A_\lambda(q,t,1) = \sum_{\tau \preceq \mu} q^{\dist(\tau, \mu)} t^{\ssim_\theta(\tau, \mu)}
\end{equation}
where the sum runs over intervals in the $\nu$-Tamari lattice of $\lambda$.
\end{conjecture} 

We have checked with computations:
\begin{itemize}
\item Up to $n=21$, for all partition $\lambda$ such that the top-down tableau of $\lambda$ is sim-sym, the polynomial~\eqref{eq:dist-sim-pol} is also symmetric. The biggest non trivial partition tested is $\lambda = (8, 6, 4, 2, 1)$ where the $\nu$-Tamari lattice has $416$ elements and $17388$ intervals.
\item For the same partitions, up to size $21$, we have checked that the polynomial can be obtained at the expansion of a Schur positive function in 2 variables (as we are missing the last variable $r$, we cannot get directly the expression as a Schur function in three variables, but the Schur positivity in $2$ variables is a necessary condition).
\item François Bergeron provided us with a list of potential $A_\lambda$ Schur functions and we checked that the polynomial~\eqref{eq:dist-sim-pol} was indeed the expansion of these functions in $q,t,1$. We have provided the list of functions on Table~\ref{tab:al-table}.
\end{itemize}

For the last point, note that these $A_\lambda$ functions are related to diagonal harmonic polynomials and are actually very difficult to compute. For some of them, the expression computed by François Bergeon is only conjectural and might be missing some terms. Actually, when first tested against our conjecture, we noted that some terms were \emph{not} equal to our enumeration when expanded on $q,t,1$. It was the case of $A_{6,4,2,1}$ and $A_{7,5,3,1}$. Looking more closely, we noticed that the expansion would be equal if some extra Schur terms were added to the formula that had been provided to us: we added $s_{3,2,2}$ to $A_{6,4,2,1}$, and $ s_{6,2,2} + s_{5,2,2} + s_{4,2,2}$ to $A_{7,5,3,1}$. Discussing with François Bergeron, we concluded that those terms were difficult to get through his computation method and conjectured that the actual $A_\lambda$ value contained them.

All the code for these tests is available on~\cite{PonSage25}. Note that we have only included in Table~\ref{tab:al-table} partitions of length greater than or equal to $3$. Indeed, for partitions of length $1$, the result is trivial, for partitions of length $2$, we have an exact formula and can prove the conjecture. First notice that the top-down tableau of a $2$-partition is always sim-sym.

\begin{proposition}
\label{prop:top-down-2parts}
Let $\lambda = (m,n)$ be a $2$-partition, then the top-down tableau of $\lambda$ is the maximal row-regular tableau of $\lambda$. 
The labels on the upper line are given by $m - n + 2, m - n + 4, \dots, m +n$. The labels on the lower line are given by $1, 2, \dots, m-n, m-n + 1, m-n+3, \dots, m+n -1$. 

In particular, the top-down tableau of $\lambda$ is sim-sym.
\end{proposition}

\begin{proof}
The top-down tableau labels the cells from both line alternatively starting from the upper line until the upper line is filled. This is the definition of the maximal row-regular tableau given in Definition~\ref{def:row-regular}. The fact that it is sim-sym is a consequence of Theorem~\ref{thm:row-regular-athqt}.
\end{proof}

We prove the following theorem in Section~\ref{subsec:proof-conj-intervals-2part} which implies in particular that Conjecture~\ref{conj:lattice} is true on $2$-partitions.

\begin{theorem}
\label{thm:qtrsym}
Let $\lambda = (m,n)$ a triangular partition, $\theta$ its top-down tableau and 

\begin{equation}
A_\lambda := \sum_{0 \leq d \leq \min(n, m-n)} s_{m+n-2d,d}.
\end{equation}

Then 
\begin{equation}
A_\lambda(q,t,1) = \sum_{\tau \preceq \mu} q^{\dist(\tau, \mu)} t^{\ssim_\theta(\tau, \mu)}
\end{equation}
summed over the intervals of the $\nu$-Tamari lattice for $\lambda$.
\end{theorem}

Note that the formula is the same as in Theorem~\ref{thm:alqt2} but now expanded on three variables instead of two (with the tjirs variable equals to $1$). We show an example on Figure~\ref{fig:intervals-31}.

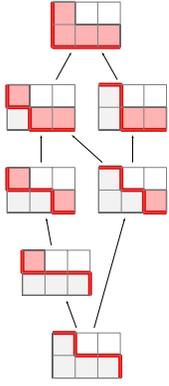
\begin{figure}[ht]
\begin{minipage}{.25\textwidth}
\scalebox{.3}{\input{figures/nutamari-example-31}}
\end{minipage}
\begin{minipage}{.7\textwidth}
\begin{align*}
A_{3,1} &= s_{4} + s_{2,1} \\
A_{3,1}(q,t,1) &= q^{4} + q^{3} t + q^{2} t^{2} + q t^{3} + t^{4} + q^{3} + 2 q^{2} t + 2 q t^{2} \\
&+ t^{3} + 2 q^{2} + 3 q t + 2 t^{2} + 2 q + 2 t + 1 \\
A_{3,1}(1,1,1) &= 23
\end{align*}
\end{minipage}
\caption{Interval enumeration for $\lambda = (3,1)$.}
\label{fig:intervals-31}
\end{figure}

\subsection{Structure of the $\nu$-Tamari lattice on $2$-partitions}

To prove Theorem~\ref{thm:qtrsym}, we need to understand better the specific structure of the $\nu$-Tamari lattice on $2$-partitions. The geometric realizations of the $\nu$-Tamari lattice has been studied in~\cite{CPS16}. From their work, we know that in the case of $2$-partitions, the lattice forms a subdivision of a pentagon. In particular, it is a polygonal lattice, its Hasse diagram is planar and made of ``glued'' $2$d-faces as we see on the example of Figure~\ref{fig:PRV}. The faces are either pentagons or squares.

\begin{figure}[ht]
    \centering
\scalebox{0.6}{\input{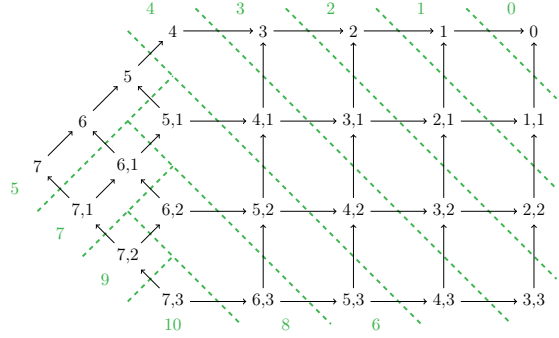}}
\caption{The structure of the $\lambda$-lattice on a $2$-partition, with $\lambda = (7,3)$. Elements share the same $\theta$-\emph{sim} between two green dashed lines.}
\label{fig:PRV}
\end{figure} 

Let us explicit the definitions and properties that are needed for this paper. For the rest of this section, $\lambda = (m,n)$ is always a triangular partition and $\theta$ is its top down tableau.

\begin{proposition}
\label{prop:mu-rotations}
Let $\mu = (m - i, n - j)$ a subpartition of $\lambda$. 

If $m - i > n - j$, then a $\nu$-Tamari rotation can be applied on the lower line of $\mu$. We call it a \defn{lower-rotation} on $\mu$ and write it $\lowrot(\mu)$. We have that $\lowrot(\mu) = (m-i-1, n -j)$.

If $n - j > 0$, then a $\nu$-Tamari rotation can be applied on the upper line of $\mu$. We call it an \defn{upper-rotation} on $\mu$ and write it $\uprot(\mu)$. We have that 
\begin{itemize}
\item if $i \leq j$, then $\uprot(\mu) = (m - i, n- j-1)$ (only one cell is removed);
\item otherwise, $\uprot(\mu) = (m-i-1, n-j-1)$ (two cells are removed).
\end{itemize}
\end{proposition}

See examples on Figure~\ref{fig:PRV} where $\lambda = (7,3)$. The upper-rotations are all cover relations going in direction north-west in the picture, for example between $(7,3)$ and $(7,2)$. You can see that they remove only one cell of the partition $\mu$. The other cover relations (going north, north-east, and east) are lower-rotations. The ones going east or north-east remove only one cell. For example, between $(7,3)$ and $(6,3)$, we have $m-m' = 7 - 7 = 0 \leq 0 = 3 - 3 = n - n'$, between $(7,2)$ and $(6,2)$, we have $m - m' = 7 - 7 = 0 \leq 1 = 3 - 2 = n - n'$. The cover relations going north are upper rotations which remove two cells. For example, from $(6,3)$ to $(5,2)$, we have $m - m' = 7 - 6 = 1 > 0 = 3 - 3 = n - n'$. 

\begin{proof}
This is a direct application of the definition of $\nu$-Tamari rotations on $2$-partitions.
\end{proof}

\begin{definition}
Let $\mu = (m-i,n-j)$ be a subpartition of $\lambda$. If one can apply both an upper and a lower rotation on $\mu$, we call the interval between $\mu$ and $\uprot(\mu) \join \lowrot(\mu)$ a \defn{polygon}.
\end{definition}

These so-called polygons are actual 2-dimensional faces in the geometric realization of the $\nu$-Tamari lattice. They also are polygons in the sense of a \emph{polygonal lattice}, \emph{i.e.}, the interval is formed by two chains which meet only at their ends. These are known properties of the $\nu$-Tamari lattice but as it is very simple in our case, we give a direct proof and a characterization of the different types of polygons that can appear.

\begin{proposition}
Let $\mu = (m-i,n-j)$ be a subpartition of $\lambda$ such that $m-i > n-j > 0$ (both upper and lower rotations can be applied at $\mu$).

If $i = j$, then we call the polygon starting at $\mu$ is a \defn{pentagon} (as it contains $5$ elements), it is formed by the union of two chains: $(m-i,n-j) \prec (m-i,n-j-1) \prec (m-i-1, n-j-1) \prec (m-i-2, n-j-1)$ and $(m-i,n-j) \prec (m-i-1, n-j) \prec (m-i - 2, n-j-1)$. 

If $ i < j$, then we call the polygon starting at $\mu$ a \defn{small square} (as it contains $4$ elements), it is formed by the union of two chains: $(m -i, n -j) \prec (m-i, n-j-1) \prec (m-i-1, n-j-1)$ and $(m-i,n-j) \prec (m-i-1, n-j) \prec (m-i-1,n-j-1)$.

If $i > j$, then we call the polygon starting at $\mu$ a \defn{big square} (as it contains $4$ elements), it is formed by the union of two chains: $(m-i, n-j) \prec (m-i-1, n-j-1) \prec (m-i-2, n-j-1)$ and $(m-i,n-j) \prec (m-i-1, n-j) \prec (m-i-2,n-j-1)$.
\end{proposition}

\begin{proof}

\begin{figure}[ht]
\input{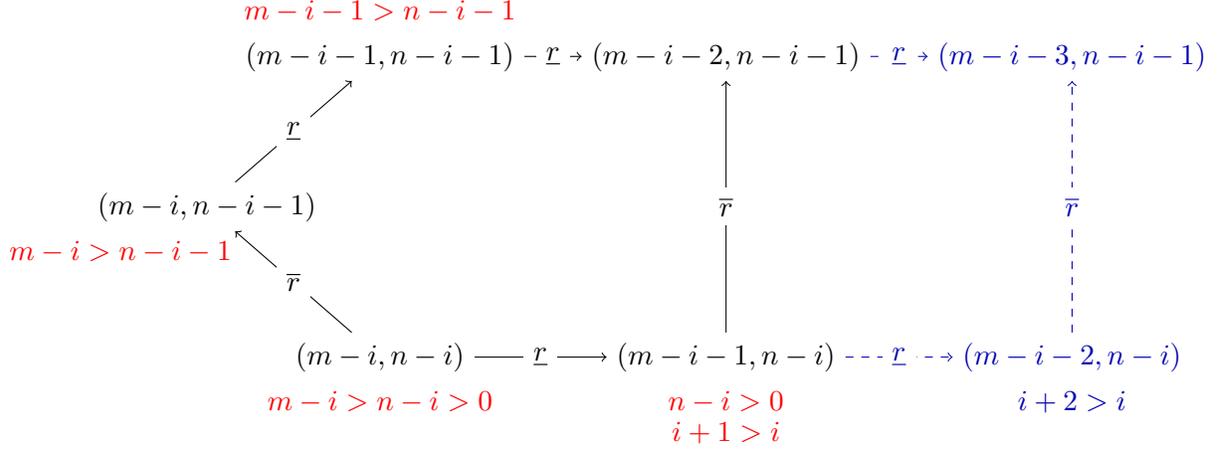}
\caption{Proof of the pentagon case.}
\label{fig:proof-pentagon}
\end{figure}

\begin{figure}[ht]
\input{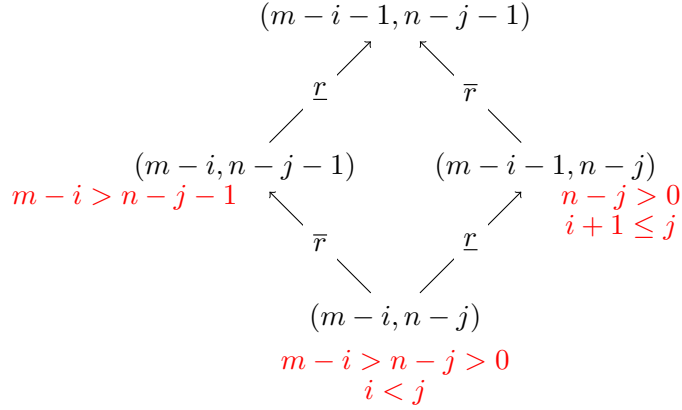}
\caption{Proof of the small square case.}
\label{fig:proof-small-square}
\end{figure}

\begin{figure}[ht]
\input{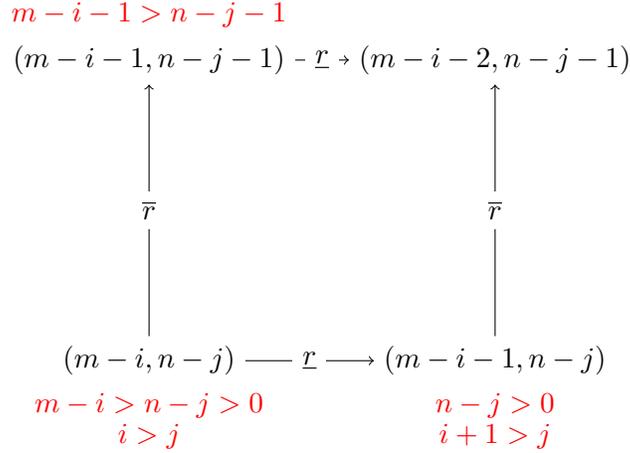}
\caption{Proof of the big square case.}
\label{fig:proof-big-square}
\end{figure}

The proof for the three different cases are given in Figures~\ref{fig:proof-pentagon} (pentagon), \ref{fig:proof-small-square} (small square), and~\ref{fig:proof-big-square} (big square). First note that as $m -i > n -j > 0$, we have $m-i \geq 2$ and $n-j \geq 1$ so all subpartitions in the polygons are well defined. Following the graph, you can see that elements of the chains are obtained through proper $\nu$-Tamari rotations as the equalities and inequalities in red are satisfied. In the small and big square cases, the lower and upper rotations commute so it is clear that the top element is indeed the join of $\uprot(\mu)$ and $\lowrot(\mu)$ and that no other rotation can be applied while staying in the interval. In the pentagon case, we use Lemma~\ref{lem:nutam-incl} to complete the proof. We have written in dashed blue the lower rotation between $(m-i-1, n-i)$ and $(m-i-2, n-j)$ (only possible if $m-i-1 > n-i$). Following Lemma~\ref{lem:nutam-incl}, $(m-i-2, n-i)$ is the only subpartition, beside the five subpartitions of the pentagon, which might belong to the interval between $(m-i,n-i)$ and $(m-i-2,n-i-1)$. We see in the figure that it is not the case as we obtain $(m-i-3,n-i-1) \succ (m-i-2, n-i-1)$ by applying an upper rotation.
\end{proof}

\begin{definition}
Let $\mu = (m-i, n-j)$ be a subpartition of $\lambda$, we say that $\mu$ is 
\begin{itemize}
\item on the \defn{left side} if $i<j$;
\item in the \defn{center} if $i = j$;
\item on the \defn{right side} if $i > j$.
\end{itemize}
\end{definition}

In the example of Figure~\ref{fig:PRV}, we have $6$ subparitions on the left side ($(7,2)$, $(7,1)$, $(7)$, $(6,1)$, $(6)$, and $(5)$), $4$ subpartitions in the center ($(7,3)$, $(6,2)$, $(5,1)$, and $(4)$) and the rest on the right side. In general, the small squares are all on the left side, while the big squares are all on the right side. We have $n+1$ elements $(m-i,n-i)$ in the center ($i = 0,1,\dots,n$). If $i < n$, then they correspond to the minimal element of a pentagon, so there are exactly $n$ pentagons. 

Now that we understand the structure, we can easily compute the distance between two elements.

\begin{proposition}
\label{prop:nu-prec-right}
Let $\tau = (\tau_1, \tau_2)$ be a subpartition of $\lambda$ such that $\tau$ is on the right side and let $\mu = (\mu_1, \mu_2)$ with $\mu \succ \tau$. Then $\mu$ is also on the right side and $\dist(\tau, \mu) = \tau_1 - \mu_1$.
\end{proposition}

For example, the distance between $(5,2)$ and $(2,1)$ in the lattice of Figure~\ref{fig:PRV} is $5 - 2 = 3$.

\begin{proof}
Let $\tau = (\tau_1, \tau_2)$ be a subpartition of $\lambda$ which is on the right side. Then $\tau_1 = m - i$ and $\tau_2 = n - j$ with $i > j$. By Proposition~\ref{prop:mu-rotations}, by applying an upper rotation, we obtain $(m-i-1, n-j-1)$ which is still on the right side. Similarly, by applying a lower rotation, we obtain $(m-i-1, n-j)$ which is also still on the right side. So any element covering $\tau$ stays on the right side and by induction also any subpartition $\mu \succ \tau$. Besides, the length of the lower part is reduced by one by any of the two rotations which means that any chain between $\tau$ and $\mu$ is of length $\tau_1 - \mu_1$. 
\end{proof}

\begin{proposition}
\label{prop:nu-prec-right-comp}
Let $\mu =(\mu_1, \mu_2)$ and $\tau = (\tau_1, \tau_2)$ be tw subpartitions of $\lambda$ such that $\tau$ is on the right side. Then $\tau \preceq \mu$ if and only if $\mu \subseteq \tau$ and $\tau_1 - \mu_1 \geq \tau_2 - \mu_2$.
\end{proposition}

\begin{proof}
Suppose that we have $\mu = (\mu_1, \mu_2)$ on the right side and $\tau$ such that $\mu \subseteq \tau$ and $\tau_1 - \mu_1 \geq \tau_2 - \mu_2$. We prove that $\tau \preceq \mu$ by induction on $\tau_1 - \mu_1$. Suppose first that $\tau_1 - \mu_1 = 0$, then we can reach $\mu$ from $\tau$ by applying a series of $\tau_2 - \mu_2$ lower-rotations which each decrease the lower part by one (and do not change the upper part) and $\tau \preceq \mu$. Now suppose that $\tau_1 - \mu_1 > 0$, we apply an upper rotation and as $\tau$ is on the right side, we obtain $\tau' = (\tau_1 -1, \tau_2-1)$ and by Proposition~\ref{prop:nu-prec-right}, $\tau'$ is still on the right side. Besides $(\tau_1 -1) - \mu_1 = (\tau_1 - \mu_1) - 1 \geq (\tau_2 - \mu_2) - 1 = (\tau_2 -1) - \mu_2$ and we obtain the result by induction.

Now suppose that $\tau \preceq \mu$. This implies in particular that $\mu \subseteq \tau$. This also means that there is a chain between $\tau$ and $\mu$ where a $\nu$-Tamari rotation is applied at each step. This chain contains exactly $\tau_2 - \mu_2$ upper rotations, because only upper rotations can decrease the upper row. As $\tau$ is on the right side, it is also the case of any subpartition in the chain (from Proposition~\ref{prop:nu-prec-right}), so all the upper rotations also decrease the lower row by one. This implies that $\tau_1 - \mu_1 \geq \tau_2 - \mu_2$. 
\end{proof}

\begin{lemma}
\label{lem:nu-chain}
Let $\mu = (\mu_1, \mu_2)$ and $\tau = (\tau_1, \tau_2)$ be two subpartitions such that $\tau$ is in the center or on the left side and $\mu \subseteq \tau$, \emph{i.e.}, $\tau_1 \geq \mu_1$ and $\tau_2 \geq \mu_2$. Then 
\begin{equation}
(\tau_1, \tau_2) \prec (\tau_1, \tau_2 - 1) \prec \dots \prec (\tau_1, \mu_2) \prec (\tau_1 - 1, \mu_2) \prec \dots \prec (\mu_1, \mu_2)
\end{equation}
is a well defined chain of the $\nu$-Tamari lattice.
\end{lemma}

For example, on Figure~\ref{fig:PRV}, take $\tau = (7,2)$ and $\mu = (3,1)$, then the chain is $(7,2) \prec (7,1) \prec (6,1) \prec (5,1) \prec (4,1) \prec (3,1)$.

\begin{proof}
Let $\tau' = (\tau_1, \mu_2)$ and prove that, when $\tau' \neq \mu$, the chain between $\tau'$ and $\mu$ is well defined. This is immediate: at each step, we remove one cell from the lower line (which is always possible because $\tau_1 > \mu_1 \geq \mu_2$) which corresponds to a lower rotation.
Now we prove that the chain between $\tau$ and $\tau'$ is well defined by induction on $\tau_1 - \mu_1$. If $\tau_1 - \mu_1 = 0$, it is the empty chain and there is nothing to prove. Now we write $\tau = (m - i, n -j)$ and $\tau' = (m-i, n-j')$ with $j < j'$. In particular, $n - j > 0$ and we can apply an upper rotation on $\tau$. As $\tau$ is in the center or on the left side, we have by hypothesis that $i \leq j$, so $\uprot(\tau) = (m - i, n - j - 1)$ which is the next element in the chain. Beside, we have $i < j + 1$, so $(m - i, n - j - 1)$ is on the left side and we can conclude the proof by induction.
\end{proof}

\begin{proposition}
\label{prop:nu-prec-left}
Let $\mu = (\mu_1, \mu_2)$ and $\tau = (\tau_1, \tau_2)$ be two subpartitions of $\lambda$ such that $\tau$ is in the center or on the left side. Then $\tau \preceq \mu$ if and only if $\mu \subseteq \tau$ and in this case $\dist(\tau, \mu) = \tau_1 - \mu_1 + \tau_2 - \mu_2$.
\end{proposition}

\begin{proof}
Note that we know from Lemma~\ref{lem:nutam-incl} that $\mu \subseteq \tau$ is a necessary condition to have $\tau \preceq \mu$. The construction of Lemma~\ref{lem:nu-chain} proves that in the case where $\tau$ is in the center or on the left side, it is also sufficient.

Let $a = \tau_1 - \mu_1 + \tau_2 - \mu_2$ which is also equals to $\area(\mu) - \area(\tau)$. Any rotation applied on $\tau$ removes at least one cell from $\tau$ (thus increasing the area). So any chain from $\tau$ to $\mu$ has a length smaller than or equal to $a$. We get $\dist(\tau, \mu) \leq a$. As the distance is defined as the length of the longest chain between $\tau$ and $\mu$, we only have to exhibit a chain of length~$a$ to prove the result. We use the chain of Lemma~\ref{lem:nu-chain}.
\end{proof}

We now use the \emph{sim} to partition the lattice into subsets. 

\begin{proposition}
\label{prop:nu-tam-sim}
Let $s$ be an integer value such that $0 \leq s \leq m+n$, we write $E(s)$ the set of subpartitions $\mu$ of $\lambda$ such that $\simthm = s$. We have the following.
\begin{enumerate}
\item If $s \leq m-n$, then $E(s) = \lbrace (s), (s,1), (s,2), \dots (s, \min(s,n)) \rbrace$. In particular $|E(s)| = \min(s,n) + 1$. 
All subpartitions of $E(s)$ are all on the right side except for the subpartition $(m-n) \in E(m-n)$ which is in the center.
\label{case-prop:nu-tam-sim-small}
\item If $s > m-n$ and is of same parity as $m+n$, then there exists an integer $i$ with $0 \leq i < n$ such that $s = m + n -2i$ and we have $E(s) = \lbrace (m-i, n-j); 0 \leq j \leq i \rbrace$. In particular, $|E(s)| = i+1$. For all $s$, the subpartition $(m-i, n-i)$ is in the center and the other subpartitions of $E(s)$ are on the right side.
\label{case-prop:nu-tam-sim-even}
\item If $s > m-n$ and of different parity than $m+n$, then there there exists an integer $j$ with $1 \leq j \leq n$ such that $s = m + n -2j + 1$ and we have $E(s) = \lbrace (m-i, n-j); 0 \leq i < j \rbrace$. In particular, $|E(s)| = j$. For all $s$, all subpartitions of $E(s)$ are on the left side.
\label{case-prop:nu-tam-sim-odd} 
\end{enumerate}
\end{proposition}

This corresponds to the partition of the lattice by the green lines on Figure~\ref{fig:PRV}. We show specific examples from this lattice on Figure~\ref{fig:2parts-ex-sim}. You can see for example that for $s = 3 < 7  - 3$, we have $4$ subpartitions $\mu$ with $\simthm = 3$ which are $(3)$, $(3,1)$, $(3,2)$, and $(3,3)$ (all on the left side).  For $s = 6$, same parity than $10$, $s = 10 - 2 \times 2$ and there are $3$ subpartitions $\mu$ with $\simthm = 6$: $(5,1)$, $(5,2)$, and $(5,3)$ (with $(5,1)$ in the center and the other on the right side). For $s = 7$, different parity than $10$, we have $7 = 10 - 2 \times 2 + 1$, we get $2$ subpartitions $\mu$ with $\simthm = 7$: $(6,1)$, and $(7,1)$ (both on the left side). 

\begin{figure}
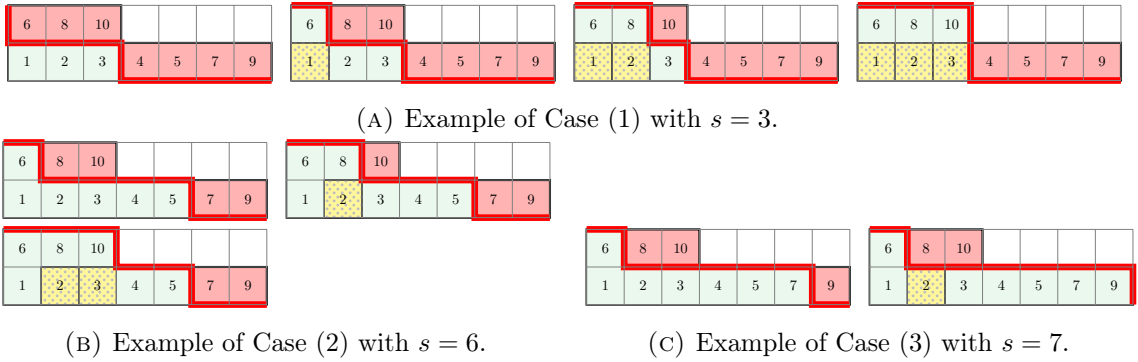

\begin{subfigure}{1 \textwidth}
\centering
\scalebox{.5}{
\input{figures/example-sim-small}
}
\caption{Example of Case~\eqref{case-prop:nu-tam-sim-small} with $s = 3$.}
\label{fig:2parts-ex-sim-small}
\end{subfigure}
\begin{subfigure}[b]{.49 \textwidth}
\centering
\scalebox{.5}{
\input{figures/example-sim-right-side}
}
\caption{Example of Case~\eqref{case-prop:nu-tam-sim-even} with $s = 6$.}
\label{fig:2parts-ex-sim-right-side}
\end{subfigure}
\begin{subfigure}[b]{.49 \textwidth}
\centering
\scalebox{.5}{
\input{figures/example-sim-left-side}
}
\caption{Example of Case~\eqref{case-prop:nu-tam-sim-odd} with $s = 7$.}
\label{fig:2parts-ex-sim-left-side}
\end{subfigure}
\caption{Examples of subpartitions with same sim values.}
\label{fig:2parts-ex-sim}
\end{figure}

\begin{proof}
For the sake of the proof, we define $E(s)$ constructively using the descriptions given in the different cases and first prove that for all $\mu \in E(s)$, $\simthm = s$.

\paragraph{Proof for Case~\eqref{case-prop:nu-tam-sim-small}.} Let $s \leq m - n$ and $E(s) := \lbrace (s,k); 0 \leq k \leq \min(s,n) \rbrace$. 
Let $\mu = (s, k)$ with $k \leq \min(s,n)$. As $k \leq s$, $\mu$ is a well defined partition. And because $k \leq n$, it is indeed a subpartition of $\lambda$. Beside, as $s = m - i$ with $i \geq n$, $\mu$ is necessary on the right side except when $i= n$ and $k=0$, then it is on the center. We also set $\tau := (s,0)$. In particular, $\tau \subseteq \mu$.

As the top-down tableau of $\lambda$ is the maximal row-regular tableau of $\lambda$ described in Proposition~\ref{prop:top-down-2parts}, the first $m-n+1$ cells of the lower row are labeled $1,2,\dots,m-n+1$. As $s \leq m-n$, this means that the cells of $\tau$ have labels $1,2\dots, s$. In particular, $\tau$ has no deficit and $\simtht = s$. Now if there is a cell in $\mu$ and not in $\tau$, it is on the upper row with label $m-n+2i$ and $0 \leq i \leq \min(s,n)$. The cell $m-n+2i+1$ is on the lower row and not in $\mu$. In other words, whenever we add a cell on the upper row of $\tau$, we add a sim cell (the new cell) but create a deficit in the cell below, thus keeping the sim value constant. See an example on Figure~\ref{fig:2parts-ex-sim-small}. We find $\simthm = \simtht = s$.

\paragraph{Proof for Case~\eqref{case-prop:nu-tam-sim-even}.} Let $s > m -n$ and with same parity than $m+n$. Then there exists a certain integer $i \geq 0$ such that  $s = m + n -2i$. As we have $s > m -n$, we obtain $i < n$. Let $E(s) := \lbrace (m-i, n-j) ; 0 \leq j \leq i \rbrace$ and $\mu = (m-i, n-j) \in E(s)$. Note that $\mu$ is well defined as the conditions on $i$ and on $2$-triangular partition implies that $m-i \geq n$. Beside $\mu$ is on the right side by definition except when $j=i$, then it is in the center.
 We also set $\tau := (m-i, n-i)$. In particular, we have $\tau \subseteq \mu$. 

The labels of the top-down tableau given in Proposition~\ref{prop:top-down-2parts} can be written as: $1, 2, \dots, m-n$ followed by $m+n-2(n-1)-1, m+n-2(n-2)-1, \dots, m+n-1$ on the lower row and $m+n-2(n-1), m+n-2(n-2), \dots, m+n$ on the upper row.  In particular, the labels of the cells in $\tau$ are $1,2\dots m-n$ followed by $m+n-2(n-1)-1,\dots, m-2i-1$ on the lower row and $m+n-2(n-1), \dots, m+n-2i$ on the upper row. This corresponds to all labels $k$ for $k \leq m+n-2i$. See for example Figure~\ref{fig:2parts-ex-sim-right-side} where $m=7$, $n=3$, and $i=2$. The subpartition $\tau$ is $(5,1)$ and you can see that it corresponds to all cells with labels smaller than or equal to $6$. It means in particular that there are no deficit in $\tau$ and $\simtht = m -i + n- i = s$. Now, if there is a cell in $\mu$ but not in $\tau$, then it is on the upper row with a label $m+n-2k$ with $k<i$, the cell with label $m+n-2k-1$ is on the lower row and not in $\mu$: it creates a deficit. In other words, whenever we add a cell on the upper row of $\tau$, we add sim cell (the new cell) but create a deficit in the cell below, thus keeping the sim value constant. We find that $\simthm = \simtht = s$.

\paragraph{Proof for Case~\eqref{case-prop:nu-tam-sim-odd}.} Let $s > m -n$ and with different parity than $m+n$. Then there exists a certain integer $j \geq 1$ such that $s = m + n -2j + 1$. As we have $s > m-n$, we find that $j \leq n$. Let $E(s) :=  \lbrace (m-i, n-j); 0 \leq i < j \rbrace$ and $\mu = (m-i, n-j) \in E(s)$. As $m \geq n$ and $i < j$, we have $m-i > m-j$ and $\mu$ is a well defined partition. It is on the left side by definition. We also set $\tau := (m-j+1, n-j)$. In particular, we have $\tau \subseteq \mu$.

Following the above description of the top-down tableau, the cells of $\tau$ are labeled by $1,2,\dots,m-n$ then $m+n-2(n-1)-1,\dots, m-2(j-1)-1$ on the lower row, and $m+n-2(n-1),\dots,m+n-2j$ on the upper row. This corresponds to all labels $k$ for $k\leq m-2(j-1)-1$. See for example Figure~\ref{fig:2parts-ex-sim-left-side} where $m=7$, $n=3$, and $j=1$. The subpartition~$\tau$ is $(6,1)$ and you can see that it corresponds to all cells with labels smaller than or equal to $7$. It means in particular that there are no deficit in $\tau$ and $\simtht = m -j+1 + n- j = s$. Now, if there is a cell in $\mu$ but not in $\tau$, then it is on the lower row with a label $m+n-2k -1$ with $k<j-1$, the cell with label $m+n-2(k+1)$ is on the upper row and not in $\mu$: it creates a deficit. In other words, whenever we add a cell on the lower row of $\tau$, we add sim cell (the new cell) but create a deficit in the cell below $m+n-2(k+1)$, thus keeping the sim value constant. We find that $\simthm = \simtht = s$.

We have proven that for each $s$, $E(s) \subseteq \lbrace \mu; \simthm = s \rbrace$. To prove the equality, we show that any $\mu = (m-i, n-j)$ belongs to exactly one $E(s)$ as defined above. Suppose first that $\mu$ is on the left side, \emph{i.e.} $i < j$ and take $s = m +n -2j +1$. As $j \leq n$, we have that $s > m-n$. It is also of different parity than $m+n$ and $\mu \in E(s)$ (Case~\eqref{case-prop:nu-tam-sim-odd}) by definition. Now suppose that $j \leq i < n$ and take $s = m + n -2i$. We have that $s > m-n$ and of same parity than $m+n$: by definition, $\mu \in E(s)$ (Case~\eqref{case-prop:nu-tam-sim-even}). Finally, suppose that $i \geq n$ and take $s = m-i \leq m -n$. As $\mu$ is a partition, $n-j \leq m -i$ and also $n-j \leq n$ by definition so $\mu = (s,k)$ with $0 \leq k \leq \min(s,n)$ and $\mu \in E(s)$ (Case~\ref{case-prop:nu-tam-sim-small}).
\end{proof}

\begin{corollary}
\label{cor:nu-tam-sim}
Let $\mu = (m-i, n-j)$ be a subpartition of $\lambda$. Then 
\begin{itemize}
\item if $\mu$ is in the center or on the right side, then $\simthm = m -i + \min(0, n-i)$ (the value depends only of $i$);
\item if $\mu$ is on the left side, then $\simthm = m + n - 2j + 1$ (the value depends only of $j$).
\end{itemize}
\end{corollary}

\begin{proof}
Suppose first that $\mu = (m-i, n-j)$ is in the center or on the right side. In particular $j \leq i$ and we are in Case~\eqref{case-prop:nu-tam-sim-small} or Case~\eqref{case-prop:nu-tam-sim-even} of Proposition~\ref{prop:nu-tam-sim}. Suppose that $n - i > 0$ and take $s = m + n -2i$. We have $s > m-n$ and $\mu \in E(s)$ following Case~\eqref{case-prop:nu-tam-sim-even}. On the other hand, if $n - i \leq 0$, we have $\mu \in E(m-i)$ following Case~\eqref{case-prop:nu-tam-sim-small} as $m-i \leq m-n$ and $n-j \leq \min(m-i, n)$.

Now if $\mu$ is on the left side, we are necessarily in Case~\ref{case-prop:nu-tam-sim-odd} and $\mu \in E(m + n -2j +1)$.
\end{proof}

\subsection{Proof of the conjecture for $2$-partitions}
\label{subsec:proof-conj-intervals-2part}

For the time being, we write

\begin{equation}
P_{m,n}(q,t) := \sum_{\tau \preceq \mu} q^{\dist(\tau, \mu)} t^{\ssim_\theta(\tau, \mu)} 
\end{equation}
summed over the intervals of the $\nu$-Tamari lattice for $\lambda = (m,n)$. Our goal is to prove Theorem~\ref{thm:qtrsym}, \emph{i.e.} that $P_{m,n}(q,t) = A_\lambda(q,t,1)$. We proceed by induction. More precisely, we are proving that for $n >0$,

\begin{equation}
P_{m,n} - P_{m-1,n-1} = A_{m,n} - A_{m-1,n-1}.
\end{equation}

In other words, both quantities grow the same way. Regrettably, our proof does not give a bijection between the intervals and the monomials of the Schur function. This shows that even in the ``simple'' case of $2$-partitions, the question of finding a bijective proof is difficult.

We start with the base case of the induction.

\begin{proposition}
\label{prop:lattice-rec-0}
Let $m \geq 0$, then 
\begin{equation}
P_{m,0}(q,t) = A_{m,0}(q,t,1) = \sum_{s=0}^m q^s \left( \sum_{d=0}^{m-s} t^d \right) 
\end{equation}

\end{proposition}

\begin{proof}
Remember that $A_{m,0} = s_m$. Now following the definition of Schur functions, we have
\begin{equation}
s_m(q,t,r) = \sum_{T \in \SSYT(m)} q^{T_1}t^{T_2}r^{T_3}
\end{equation}
summed over the semi-standard Young tableaux of shape $m$ and values in $\lbrace 1,2,3 \rbrace$, where $T_i$ is the number of occurrences of $i$ in $T$. The tableau $T$ is made of single horizontal line. As it is semi-standard, the line is filled by a certain number of $1$, followed by a certain number of $2$ and the rest filled with $3$. Let $s$ be the number of $1$: it varies between $0$ and $m$. For each $s$, you can have between $0$ and $m-s$ values $2$ in the tableau. This gives the expected formula.

Now the $\nu$-Tamari lattice for $\lambda = m$ is just the total order $m \prec m-1 \prec \dots \prec 0$. If $\mu$ is subpartition of $\lambda$, then $\mu = m - i$ and $\simthm = m-i$. Any subpartition $\tau \preceq \mu$ is of the form $m-j$ with $j \leq i$, the distance between $\tau$ and $\mu$ is $i - j$. We find that

\begin{equation}
P_{m,0}(q,t) = \sum_{i=0}^m q^{m-i} \sum_{j=0}^i t^{i-j}
\end{equation} 
which gives the expected formula by a simple change of variable.
\end{proof}

Now, for $n > 0$, our goal is to compute $P_{m,n} - P_{m-1,n-1}$. First notice that if $m,n$ is a triangular partition and $n>0$, then $m-1,n-1$ is also a triangular partition. Besides the $\nu$-Tamari lattice for $\lambda = (m,n)$ ``contains'' the $\nu$-Tamari lattice for $m-1, n-1$. The $(m,n)$ lattice can be obtained by ``adding a line'' under the $m-1,n-1$ lattice as illustrated on Figure~\ref{fig:lattice-rec}. We formalize this result in the following proposition.

\begin{proposition}
\label{prop:lattice-rec}
Let $\lambda = (m,n)$ a triangular partition such that $n > 0$, and $\tau \preceq \mu$ an interval in the $\nu$-Tamari lattice of $\lambda$ with $\tau = (m-i, n-j)$ and $i,j \geq 1$. Then $\tau$ and $\mu$ are subpartitions of $(m-1,n-1)$, they form an interval in the $\nu$-Tamari lattice of $(m-1,n-1)$ and the \emph{sim} and \emph{distance} of the interval is the same in both lattices.
\end{proposition}

\begin{figure}[ht]
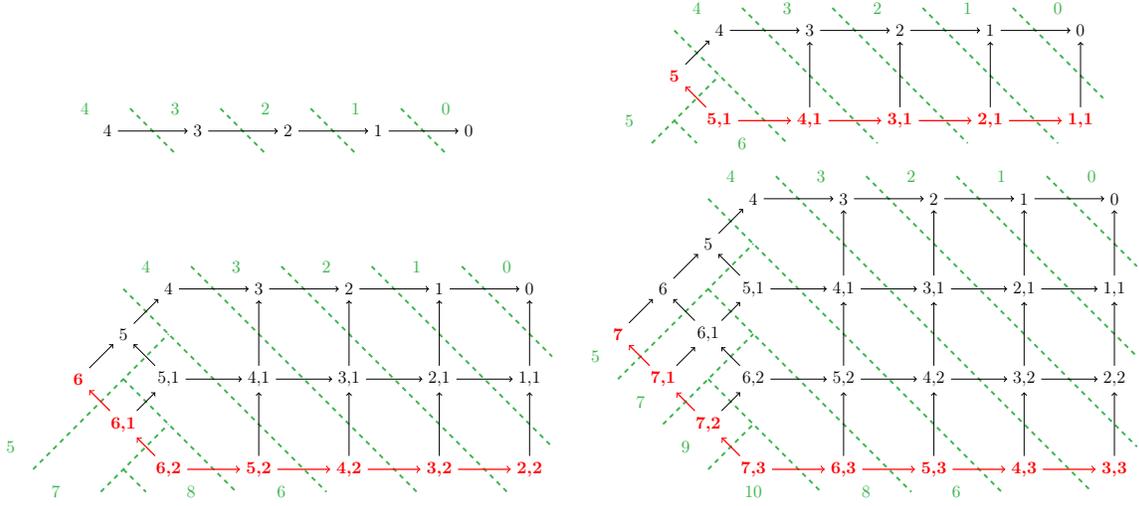

\begin{tabular}{cc}
\scalebox{0.6}{\input{figures/lattice_rec_4-0}}
&
\scalebox{0.6}{\input{figures/lattice_rec_5-1}} \\
\scalebox{0.6}{\input{figures/lattice_rec_6-2}}
&
\scalebox{0.6}{\input{figures/lattice_rec_7-3}}
\end{tabular}
\caption{The recursive construction of $\nu$-Tamari lattices for $2$-partitions}
\label{fig:lattice-rec}
\end{figure}

\begin{proof}
Let us call $L_1$ the $\nu$-Tamari lattice of $\lambda = (m,n)$ and $L_0$ the $\nu$-Tamari lattice of $(m-1,n-1)$.
We have $\tau = (m-i, n-j)$ and by hypothesis $i \geq 1$ and $j \geq 1$. This means that $\tau$ is a subpartition of $(m-1,n-1)$. As $\mu$ is included in $\tau$, this is also the case for $\mu$. So $(\tau, \mu)$ is an interval of $L_0$. Now suppose that $i < j$ (\emph{i.e.}, $\mu$ is on the left side in $L_1$), this implies that $i-1 < j -1$ and $\tau$ is also on the left side in $L_0$. The same reasoning also applies for the center and the right side and also for $\mu$: whatever sides $\tau$ and $\mu$ are in $L_1$, they are still on these same sides in $L_0$. In particular, by applying Propositions~\ref{prop:nu-prec-left} and~\ref{prop:nu-prec-right} the distance between $\tau$ and $\mu$ is the same in both $L_0$ and $L_1$. 

Now, we write $\mu = (m-i', n-j') = (m-1 - (i'-1), n-1 - (j'-1))$. Suppose first that $\mu$ is on the left side. By applying Corollary~\ref{cor:nu-tam-sim}, we obtain that $\simthm = m + n -2j' +1$ in $L_1$ and $\simthm = m - 1 + n-1 -2(j'-1) + 1$ in $L_0$: these two values are equal. If $\mu$ is in the center or on the right side, either $n - i' > 0$ and so is $(n-1) - (i'-1)$. In this case, we have $\simthm = m + n -2i'$ in $L_1$ and $\simthm = (m-1) + (n-1) -2(i'-1)$ in $L_0$. Or $n - i' \leq 0$ and we have $\simthm = m - i'$ in $L_1$ and $\simthm = m - 1 - (i'-1)$ in $L_0$. In all cases, the \emph{sim} of $\mu$ is the same relatively to $(m,n)$ and $(m-1,n-1)$. 
\end{proof}

This proposition implies that the difference between $P_{m,n}$ and $P_{m-1,n-1}$ can be computed by enumerating intervals $(\tau, \mu)$ where at least one of the subpartitions is not a subparition of $(m-1, n-1)$. As $\mu$ is included in $\tau$, these are all intervals where $\tau = (m-i, n)$ or $\tau = (m, n-j)$, \emph{i.e.}, $\tau$ is on the ``red line'' of the $\nu$-Tamari lattice as shown on Figure~\ref{fig:lattice-rec}. This tells us that for all $\mu$ of the lattice, we need to enumerate the intervals $(\tau, \mu)$ with $\tau $ on the red line depending on the distance between $\tau$ and $\mu$. We define the following polynomial depending on a subpartition $\mu$.

\begin{definition}
Let $\mu$ be a subpartition of $\lambda$, then
\begin{equation}
f_\mu(t) = \sum_{\substack{\tau = (m-i', n-j') \preceq \mu \\ i' = 0 \text{ or } j' = 0}} t^{\dist(\tau,\mu)}.
\end{equation}
\end{definition}

\begin{lemma}
\label{lem:lattice-rec-proof}
Let $\mu = (m-i, n-j)$ be a subpartition of $\lambda = (m,n)$, 
\begin{itemize}
\item if $\mu$ is on the left side, then 

\begin{equation}
\label{eq:lattice-rec-proof-left}
f_\mu(t) = \sum_{d=i}^{i+j} t^d;
\end{equation}

\item if $\mu$ is on the right side, then 

\begin{equation}
\label{eq:lattice-rec-proof-right}
f_\mu(t) = \sum_{d=j}^{i+j} t^d;
\end{equation}

\item and if $\mu$ is in the center

\begin{equation}
\label{eq:lattice-rec-proof-center}
f_\mu(t) = \sum_{d=i}^{i+j} t^d = \sum_{d=j}^{i+j} t^d.
\end{equation}

\end{itemize}
\end{lemma}

\begin{proof}
We split the sum into two parts,

\begin{equation}
\sum_{\substack{\tau = (m-i', n-j') \preceq \mu \\ i' = 0 \text{ or } j' = 0}} t^{\dist(\tau,\mu)} =
\sum_{\substack{\tau = (m-i', n) \preceq \mu \\ i' > 0}}  t^{\dist(\tau,\mu)} +
\sum_{\tau = (m, n-j') \preceq \mu}  t^{\dist(\tau,\mu)}.
\end{equation}

Let us look first at the second part of the sum, \emph{i.e.}, the subpartitions $\tau \preceq \mu$ where $\tau = (m, n-j')$. In particular, $\tau$ is on the left side or in the center. Then by Proposition~\ref{prop:nu-prec-left}, we have $\tau \preceq \mu$ if and only if $\mu \subseteq \tau$, \emph{i.e.}, $j' \leq j$. Besides, we have $\dist(\tau, \mu) = i + j - j'$. We get

\begin{equation}
\label{eq:lattice-rec-proof-right-left}
\sum_{\tau = (m, n-j') \preceq \mu}  t^{\dist(\tau,\mu)} = \sum_{j'=0}^j t^{i+j-j'} = \sum_{d=i}^{i+j} t^d.
\end{equation}

We now look at the first part of sum, \emph{i.e.}, the subpartitions $\tau \preceq \mu$ where $\tau = (m-i',n)$ with $i' > 0$. In particular, $\tau$ is on the right side. By Proposition~\ref{prop:nu-prec-right}, this implies that $\mu$ also is on the right side. In other words, if $\mu$ is on the left side or in the center, this sum is equal to $0$ as there is no such subpartition $\tau$. We thus obtain~\eqref{eq:lattice-rec-proof-left} and also~\eqref{eq:lattice-rec-proof-center} because if $\mu$ is in the center then $i=j$ which gives the second part of the equality.

If $\mu$ is on the right side, then by Proposition~\ref{prop:nu-prec-right-comp}, we have $\tau \leq \mu$ if and only if $i' \leq i$ (as $\mu \subseteq \tau$), and $j \leq i - i'$ (as $\tau_1 - \mu_1 \geq \tau_2 - \mu_2$). This gives that $1 \leq i' \leq i - j$. Besides, by Proposition~\ref{prop:nu-prec-right}, we have $\dist(\tau, \mu) = i - i'$ and we get

\begin{equation}
\label{eq:lattice-rec-proof-right-right}
\sum_{\substack{\tau = (m-i', n) \preceq \mu \\ i' > 0}}  t^{\dist(\tau,\mu)} = \sum_{i' = 1}^{i-j} t^{i-i'} = \sum_{d = j}^{i-1} t^d.
\end{equation}

By summing~\eqref{eq:lattice-rec-proof-right-left} and~\eqref{eq:lattice-rec-proof-right-right}, we get the expected formula~\eqref{eq:lattice-rec-proof-right}
\end{proof}

You can check this result on different examples on Figure~\ref{fig:lattice-rec}. See the subpatition $(6,1)$ of $(7,3)$ which is on the left side with $i=1$ and $j=2$. All supartitions $\tau \preceq \mu$ are on the left side or in the center. Three of them are on the red line: $(7,1)$ at distance $1$, $(7,2)$ at distance $2$, and $(7,3)$ at distance $3$. We get indeed $t + t^2 + t^3$ which is the sum of $t^d$ for $d = i = 1$ up to $d = i+j = 3$.

Now take $\mu = (2,1)$ which is on the right side with $i=5$ and $j=2$. We look at the subpartions $\tau \preceq \mu$ on the red line. On the right side, we have $(4,3)$, $(5,3)$, and~$(6,3)$ with respective distances $2$, $3$, and~$4$. Then \eqref{eq:lattice-rec-proof-right-right} gives~$t^2 + t^3 + t^4$. On the left side or in the center, we have~$(7,1)$, $(7,2)$ and~$(7,3)$ with respective distances~$5$, $6$, and~$7$. Then~\eqref{eq:lattice-rec-proof-right-left} gives~$t^5 + t^6 + t^7$. By summing both, we get~\eqref{eq:lattice-rec-proof-right} $t^2 +  \dots + t^7$ which is indeed the sum of~$t^d$ for~$d$ between $j=2$ and~$i+j=7$.

\begin{proposition}
\label{prop:latice-rec-diff}
Let $\lambda = (m,n)$ be a triangular partition with $n > 0$, then

\begin{align}
\label{eq:lattice-rec-diff-small}
P_{m,n} - P_{m-1, n-1} &= \sum_{s=0}^{m-n} \left( \sum_{k=0}^{\min(s,n)} \sum_{d = n-k}^{m+n-k-s} t^d \right) q^s \\
\label{eq:lattice-rec-diff-even}
&+ \sum_{i=0}^{n-1} \left( \sum_{j=0}^i \sum_{d=j}^{i+j} t^d \right) q^{m+n-2i} \\
\label{eq:lattice-rec-diff-odd}
&+ \sum_{j=1}^n \left( \sum_{i=0}^{j-1} \sum_{d=i}^{i+j} t^d \right) q^{m+n-2j+1}
\end{align}
\end{proposition}

\begin{proof}
Let $(\tau, \mu)$ be an interval of the $\nu$-Tamari lattice for $\lambda = (m,n)$. If $\tau$ is a subpartition of $(m-1, n-1)$, so is $\mu$ and the monomial $q^{\simthm} t^{\dist(\tau, \mu)}$ of $P_{m,n}$ is canceled by the same monomial appearing in $P_{m-1,n-1}$ by Proposition~\ref{prop:lattice-rec}. We obtain

\begin{align}
P_{m,n} - P_{m-1, n-1} &= \sum_{\mu \subseteq \lambda } \sum_{\substack{\tau \preceq \mu \\ \tau = (m-i', n-j') \\ i' = 0 \text{ or } j' = 0}} q^{\simthm} t^{\dist(\tau, \mu)}, \\
&= \sum_{s=0}^{m+n} \left( \sum_{\mu \in E(s)}f_\mu(t) \right) q^s.
\end{align}

Now we split this sum into three parts depending on the value of $s$ following the three cases of Proposition~\ref{prop:nu-tam-sim}.

\begin{align}
\label{eq:lattice-rec-diff-small-proof}
P_{m,n} - P_{m-1, n-1} &= \sum_{s=0}^{m-n} \left( \sum_{\mu \in E(s)} f_\mu(t) \right) q^s 
\\ &+ \sum_{i=0}^{n-1} \left( \sum_{\mu \in E(m+n-2i)} f_\mu(t) \right) q^{m+n-2i} \label{eq:lattice-rec-diff-even-proof}
\\ &+ \sum_{j=1}^n \left( \sum_{\mu \in E(m+n-2j+1)} f_\mu(t) \right) q^{m+n-2j+1}.\label{eq:lattice-rec-diff-odd-proof}
\end{align}

Now we need to prove that these three sums given on~\eqref{eq:lattice-rec-diff-small-proof}, \eqref{eq:lattice-rec-diff-even-proof}, and~\eqref{eq:lattice-rec-diff-odd-proof} correspond respectively to the sums given on~\eqref{eq:lattice-rec-diff-small}, \eqref{eq:lattice-rec-diff-even}, and~\eqref{eq:lattice-rec-diff-odd}.

\paragraph{Case~\ref{case-prop:nu-tam-sim-small}.} Let us take $s \leq m-n$, then $\mu \in E(s)$ gives that $\mu = (s,k)$ with $0 \leq k \leq \min(s,n)$. Let us set $i:= m-s$ and $j := n-k$ so that $\mu = (m-i, n-j)$. We have that $\mu$ is either on the right side or in the center and by applying Lemma~\ref{lem:lattice-rec-proof}, we get

\begin{equation}
f_\mu(t) = \sum_{d=j}^{i+j} t^d = \sum_{d=n-k}^{m+n-s-k} t^d.
\end{equation}
which gives

\begin{align}
\sum_{s=0}^{m-n} \left( \sum_{\mu \in E(s)} f_\mu(t) \right) q^s &= \sum_{s=0}^{m-n} \left( \sum_{k=0}^{\min(s,n)} f_{(s,k)}(t) \right) q^s \\
&= \sum_{s=0}^{m-n} \left( \sum_{k=0}^{\min(s,n)} \sum_{d=n-k}^{m+n-k-s}t^d \right) q^s.
\end{align}

\paragraph{Case~\ref{case-prop:nu-tam-sim-even}.} We now take $s = m+n-2i$ with $0 \leq i < n$. By Proposition~\ref{prop:nu-tam-sim}, $\mu \in E(s)$ if and only if $\mu = (m-i, n-j)$ with $0 \leq j \leq i$. In particular, $\mu$ is in the center or on the right side and by applying Lemma~\ref{lem:lattice-rec-proof}, we get

\begin{align}
\sum_{i=0}^{n-1} \left( \sum_{\mu \in E(m+n-2i)} f_\mu(t) \right) q^{m+n-2i} 
&= \sum_{i=0}^{n-1} \left( \sum_{j=0}^i f_{(m-i,n-j)}(t) \right) q^{m+n-2i} \\
&= \sum_{i=0}^{n-1} \left( \sum_{j=0}^i \sum_{d=j}^{i+j} t^d \right) q^{m+n-2i}.
\end{align}

\paragraph{Case~\ref{case-prop:nu-tam-sim-odd}.} We now take $s = m+n-2j+1$ with $1 \leq j \leq n$. By Proposition~\ref{prop:nu-tam-sim}, $\mu \in E(s)$ if and only if $\mu = (m-i, n-j)$ with $0 \leq i < j$. In particular, $\mu$ is on the left side and by applying Lemma~\ref{lem:lattice-rec-proof}, we get

\begin{align}
\sum_{j=1}^n \left( \sum_{\mu \in E(m+n-2j+1)} f_\mu(t) \right) q^{m+n-2j+1}
&= \sum_{j=1}^n \left( \sum_{i=0}^{j-1} f_{(m-i,n-j)}(t) \right) q^{m+n-2j+1} \\
&= \sum_{j=1}^n \left( \sum_{i=0}^{j-1} \sum_{d=i}^{i+j} t^d \right) q^{m+n-2j+1}.
\end{align}

\end{proof}

Our goal is now to compute $A_{m,n} - A_{m-1,n-1}$ for all triangular partition $\lambda = (m,n)$ such that $n > 0$. The formula for $A_\lambda$ is given on Theorem~\ref{thm:qtrsym}. It is a sum of Schur functions which can be obtained by a certain enumeration of semi-standard Young tableaux as expressed in~\eqref{eq:schur}. We introduce a few notations which are going to be useful.

\begin{definition}
We write $\SB_{m,n}$ the set of semi-standard Young tableau of shape $m,n$ whose values are taken in $\lbrace 1,2,3 \rbrace$. If $\pi \in \SB_{m,n}$,  we write $\pi_i$ the number of occurrences of the value $i$ in the tableau.

We write

\begin{equation}
\AB_{m,n} := \bigcup_{d = 0}^{\min(n,m-n)} \SB_{m+n-2d, d}.
\end{equation}

\end{definition}

In particular, we have

\begin{equation}
A_{m,n}(q,t,1) = \sum_{\pi \in \AB_{m,n}} q^{\pi_1} t^{\pi_2}.
\end{equation}

The first step of our computation is to show that some terms of $A_{m,n} - A_{m-1,n-1}$ cancel just like in $P_{m,n} - P_{m-1,n-1}$.

\begin{proposition}
\label{prop:new}
Let $\lambda = (m,n)$ be a triangular partition such that $n > 0$. Let $\pi \in \AB_{m-1,n-1}$ a tableau of shape $(b,a)$, we write $\varphi(\pi)$ the tableau of shape $(b+2,a)$ such that the labels of the $2$ new cells are $3$. Then $\varphi(\pi) \in \AB_{m,n}$. In other words, $\varphi(\AB_{m-1,n-1}) \subseteq \AB_{m,n}$. Besides, this is a strict inclusion and $\pi' \in \varphi(\AB_{m-1,n-1})$ if and only if $\pi'$ is of shape $m+n-2d,d$ with $d \leq n-1$ and the last two cells of the lower row of $\pi'$ are labeled by $3$. We write $\NewB_{m,n} = \AB_{m,n} \setminus \varphi(\AB_{m-1,n-1})$ and call \defn{new} tableaux the elements of $\NewB_{m,n}$.
\end{proposition}

\begin{proof}
Let $\pi \in \AB_{m-1,n-1}$. Then $\pi$ is of shape $((m -1) + (n -1) -2d,d)$ with $d \leq \min(n-1,m-1 - n+1) = \min(n-1,m-n)$. We have that $\varphi(\pi)$ is of shape $m+n-2d,d$ and as $d \leq n-1 < n$ and $d \leq m-n$, we obtain $\varphi(\pi) \in \AB_{m,n}$. 

Now take $\pi' \in \AB_{m,n}$. Then $\pi'$ is of shape $(m+n-2d, d)$ with $0 \leq d \leq \min(n,m-n)$. Suppose that $d \leq n-1$. As we have also $d \leq m-n$, we get that $d \leq \min(n-1, (m-1) - (n-1))$ and so by Remark~\ref{rem:ald}, $m+n-2-d,d$ is a well defined partition. It means that we can remove two cells from the lower row of $\pi'$ and obtain a tableau $\pi$ of $\AB_{m-1,n-1}$. If these two cells were labeled by $3$, we have $\pi' = \varphi(\pi)$.

On the other hand, if the last two cells of the lower row of $\pi'$ are not labeled by $3$, by definition it is not the image of any tableau by $\varphi$. Besides, if $\pi'$ is of shape $m+n-2d,d$ with $d > n -1$ (the only possility then is $d=n$), by removing two cells on the lower row, we obtain a shape that is not in $\AB_{m-1,n-1}$ as the condition on $d$ is not satisfied. 
\end{proof}

By definition if $\pi' = \varphi(\pi)$ we have that $\pi'_1 = \pi_1$ and $\pi'_2 = \pi_2$ so the term $q^{\pi'_1}t^{\pi'_2}$ in $A_{m,n}(q,t,1)$ is canceled by the term $q^{\pi_1}t^{\pi_2}$ in $A_{m-1,n-1}(q,t,1)$. So the difference $A_{m,n} - A_{m-1,n-1}$ in $q,t,1$ can be computed by enumerating the tableaux of $\NewB_{m,n}$.

In the following, we always assume that $(m,n)$ is a triangular partition and that $n >0$.

\begin{definition}
\label{def:psi}
Let $\pi$ be a new tableau of $\AB_{m,n}$. We say that a cell $c$ of $\pi$ is \defn{raisable} if it is on the lower row, labeled by $2$ and with at least $n$ cells to its left. We write $r(\pi)$ the number of raisable cells of $\pi$ and $R$ the subset of $\NewB_{m,n}$ formed by the tableaux containing at least one raisable cell. If $\pi$ is in $R$, we define $\psi(\pi)$ in following way. 

\begin{enumerate}
\item If the lower row of $\pi$ does not contain any $3$: we replace the last cell of the lower row by a $3$.
\label{def-case:psi-no3}
\item If the lower row contains a $3$ and if $\pi$ is of shape $m-n,n$, we replace the last raisable cell (on the lower row) by a $3$.
\label{def-case:psi-dmax}
\item If the lower row contains a $3$ and if $\pi$ is not of shape $m-n,n$, let $(b,a)$ be the shape of $\pi$, then $\psi(\pi)$ is the tableau of shape $(b-2,a+1)$ such that the extra cell on the upper row is labeled by $3$ and all the other remaining labels are kept unchanged.
\label{def-case:psi-33}
\end{enumerate}
\end{definition}

We show some examples on Figure~\ref{fig:example-psi} applying $\psi$ on some elements of $\NewB_{7,2}$. On the first example, the lower row does not contain any $3$: we are in Case~\eqref{def-case:psi-no3} and we replace the last $2$ by a $3$. On the second example, the lower row contains some $3$ but we the shape of $\pi$ is $(5,2) = (m-n, n)$: we are in Case~\eqref{def-case:psi-dmax}  
and replace the last $2$ by a $3$. On the third example, the lower row contains one $3$ and the shape is not $(5,2)$. We are in Case~\eqref{def-case:psi-33}, we remove the last two cells of the lower row ($2$ and $3$) and add a $3$ on the upper row. Finally, on the fourth example, the tableau does not contain any raisable cell as the only $2$ on the lower row has only one cell to its left. More examples are shown on Figure~\ref{fig:ex-orbit}.

\begin{figure}[ht]
\input{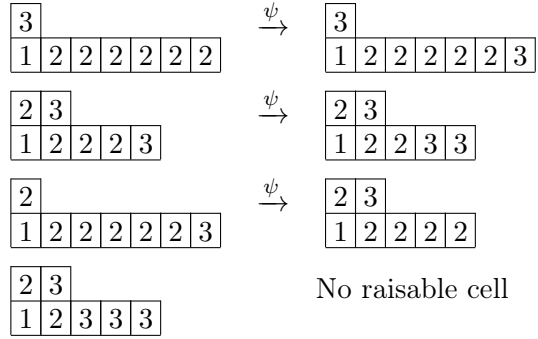}
\caption{Example of $\psi$ on some tableaux of $\NewB_{7,2}$}
\label{fig:example-psi}
\end{figure}

\begin{proposition}
\label{prop:psi}
Let $\pi$ be tableau of $\NewB_{m,n}$. Then $\pi' = \psi(\pi)$ is also in $\NewB_{m,n}$, and we have $r(\pi') = r(\pi) - 1$ and $\pi'_2 = \pi_2 - 1$.
\end{proposition}

\begin{proof}
First note that $\pi$ is of shape $m+n-2d, d$ with $0 \leq d \leq \min(n,m-n)$. As a raisable cell has at least $n$ cells to its left and $d \leq n$, it means that a raisable cell never has a cell above it.

\paragraph{Case~\eqref{def-case:psi-no3}.} Suppose first that the lower row of $\pi$ does not contain any $3$. As $\pi$ contains at least one raisable cell, it contains some cells labeled by $2$ on its lower row. In particular the last cell of the lower row is labeled by $2$ and $\pi'$ is obtained by replacing this label by a $3$. The tableau $\pi'$ is of the same shape than $\pi$ and is still a semi-standard Young tableau so $\pi' \in \AB_{m,n}$. Besides, $\pi'$ has only one cell labeled by $3$ on its lower row (the one we have just created) so $\pi' \in \NewB_{m,n}$. As we have replaced a $2$ by a $3$, we clearly have $\pi'_2 = \pi_2 -1$. The cell we have changed to $3$ was by hypothesis a raisable cell and the other cells have not been changed so the number of raisable cells is reduced by one.

\paragraph{Case~\eqref{def-case:psi-dmax}.} Suppose now that $\pi$ is of shape $m-n,n$. In this case, $\pi'$ is also of the same shape than $\pi$ and is still a semi-standard Young tableau. Indeed, the cell whose label we have changed to $3$ has no cell above it as it is raisable and is followed by cells also labeled by $3$. So $\pi' \in \AB_{m,n}$. Besides, as $\pi'$ is of shape $m+n-2d,d$ with $d=n > n-1$, we have $\pi' \in \NewB_{m,n}$. As we have changed a $2$ into a $3$, we clearly have $\pi'_2 = \pi_2 - 1$ and as before the number of raisable cells has been decreased by one.

\paragraph{Case~\eqref{def-case:psi-33}.} Now suppose that $\pi$ is of shape $m+n-2d, d$ with $d \neq n$, \emph{i.e.}, $d \leq n - 1$ and contains at least a $3$ on its lower row. Let us call $y$ and $z$ the last two cells of the lower row. By hypothesis, $z$ is labeled by $3$. Now if $y$ is labeled by $3$, then the last two cells are and $\pi \not\in \NewB_{m,n}$ as $d \leq n-1$ by hypothesis. If $y$ is labeled by $1$, $\pi$ does not contain any raisable cell. So $y$ is labeled by $2$ and is raisable. In particular, it has at least $n$ cells to its left which gives that $m+n-2d \geq n+2$. 

The tableau $\pi'$ is obtained by removing $y$ and $z$ and by adding an extra cell $y'$ on the upper line labeled by $3$. As $y$ is raisable by hypothesis, it contains at least $n$ cells to its left, all with labels smaller than or equal to $2$. As $d < n$, $y'$ is added above a cell to the left of $y$ and we obtain a well defined semi-standard Young tableau of shape $(m+n-2(d+1),d+1)$. By hypothesis, we have that $d+1 \leq n$. Suppose that $d+1 > m-n$. As $d \leq \min(n,m-n)$, this implies that $\min(n,m-n) = m-n$ and $d = m-n$. As $(m,n)$ is triangular, this can happen only if $m$ is equal to $2n - 1$ and so $m-n = n - 1$. We find that $m+n-2d = n+1$ which contradicts the fact that $y$ is a raisable cell. 

We then have $d+1 \leq m-n$. This gives that $d+1 \leq \min(n,m-n)$ and so $\pi' \in \AB_{m,n}$. More over, as no cell on the lower row of $\pi'$ is labeled by $3$, we have $\pi' \in \NewB_{m,n}$. We have removed exactly one cell labeled by $2$ which was also raisable so the result is proved.
\end{proof}

\begin{proposition}
\label{prop:psi-bij}
We call $R'$ the subset of $\NewB_{m,n}$ formed by the tableaux containing at least one value $3$. Let $\pi' \in R'$, we define $\psi'(\pi')$ as follows:

\begin{itemize}
\item if $\pi'$ contains at least one $3$ on its lower row, replace the first value $3$ with a $2$;
\item otherwise, if $\pi'$ is of shape $(b,a)$, then $\psi'(\pi')$ is the tableau of shape $(b+2,a-1)$ where the last cell of the upper row has been removed and two cells labeled by $2$ and $3$ have been added on the lower row.
\end{itemize}

Then $\psi' = \psi^{-1}$. In other words, $\psi$ is a bijection between $R$ and $R'$.
\end{proposition}

\begin{proof}
As $\psi$ always adds a cell labeled by $3$, it is clear that $\psi(R) \subseteq R'$. It is also clearly injective. Indeed if $\psi$ changes a $2$ into a $3$ on the lower row, then $\psi'$ changes it back into a $2$. If $\psi$ removes two cells from the lower row and adds a $3$ on the upper row, then we have seen that all the remaining cells of the lower row are labeled by $1$ or $2$ and that the two removed cells were labeled by $2$ and $3$: $\psi'$ reverse the operation.

We have to show that $\psi$ is surjective and that its inverse is $\psi'$. Let $\pi'\in R'$. In particular, $\pi' \in \NewB_{m,n}$ and is of shape $m+n - 2d, d$ with $d \leq \min(n,m-n)$. We write $\pi := \psi'(\pi)$, in each of the cases below, we show that $\pi$ belongs to $R$ and that $\psi(\pi) = \pi'$ following the different cases of Definition~\ref{def:psi}.

\paragraph{Goal: $\pi$ belongs to Case~\eqref{def-case:psi-no3}.} Suppose first that $\pi'$ contains at least one $3$ on its lower row and that $d \leq n - 1$. As $\pi'$ is a new tableau, by Proposition~\ref{prop:new}, it contains exactly one $3$ on its lower row: otherwise the last two cells would be labeled by $3$. So $\pi$ is the tableau of the same shape than $\pi'$ where this single $3$ has been replaced by a $2$. In particular $\pi \in \AB_{m,n}$. The last cell of the lower row is now labeled by $2$ and we claim that it is raisable. Indeed as we have $d \leq n - 1$, we have $m+n-2d \geq m - n +2$. Besides, as $(m,n)$ is triangular, we have by Remark~\ref{rem:triang} that $m-n \geq n - 1$ and we obtain $m+n-2d \geq n +1$. So the last cell of the lower row has at least $n$ cells to its left. We have that $\pi$ is in $R$. Besides, it does not contain any $3$ on its lower row so we are in Case~\eqref{def-case:psi-no3} and its image by $\psi$ is indeed $\pi'$.

\paragraph{Goal: $\pi$ belongs to Case~\eqref{def-case:psi-dmax}} Suppose now that $\pi'$ contains at least one $3$ on its lower row but is of shape $m+n-2d,d$ with $d > n-1$. The only possibility is $d = n$ and so $\pi'$ is of shape $m-n,n$. We call $y$ the first cell of the lower row of $\pi'$ labeled by a $3$. Note that as $\pi'$ is a semi-standard Young tableau with labels in $\lbrace1,2,3 \rbrace$, this means that $y$ has no cell above it. As $\pi'$ is of shape $m-n,n$, this implies that $y$ has at least $n$ cells to its left. Now $\pi$ is the tableau where we replace the label of $y$ by a $2$. As $\pi$ has the same shape that $\pi'$, clearly $\pi \in \AB_{m,n}$. As it is of shape $m+n-2d,d$ with $d > n-1$, it is also in $\NewB_{m,n}$. Finally, the cell $y$ is raisable in $\pi$ so $\pi \in R$ and we are in Case~\ref{def-case:psi-dmax} with $\psi(\pi) = \pi'$.

\paragraph{Goal: $\pi$ belongs to Case~\ref{def-case:psi-33}.} We now suppose that $\pi'$ does not contain any $3$ on its lower row. As $\pi' \in \AB_{m,n}$ by hypothesis, in particular $\pi'$ is of shape $m+n-2d,d$ with $d \leq \min(n,m-n)$ and as it contains at least one cell in its upper row, then $d > 0$. Then $\pi$ is of shape $m+n-2(d-1),d-1$ and $0 \leq d -1 < d \leq \min(n,m-n)$ so $\pi \in \AB_{m,n}$. The two cells added on the lower row in $\pi$ are labeled by $2$ and $3$ so the lower row does not end with 2 values $3$ and $\pi \in \NewB_{m,n}$. Let us call $y$ the added cell labeled by $2$. By construction it has $m+n-2d$ cells to its left. As $d \leq n$ and $d \leq m-n$, we write $m+n-2d = m+n-d-d \geq m+n-n-(m-n) = n$. The cell $y$ has at least $n$ cells to its left and is labeled by $2$: it is raisable and $\pi \in R$. Besides, $\pi$ contains a $3$ on its lower row by construction and is not of shape $m-n,n$ because $d-1 < n$: we are in Case~\ref{def-case:psi-33} and $\psi(\pi) = \pi'$.
\end{proof}

\begin{figure}[ht]
\scalebox{.6}{
\input{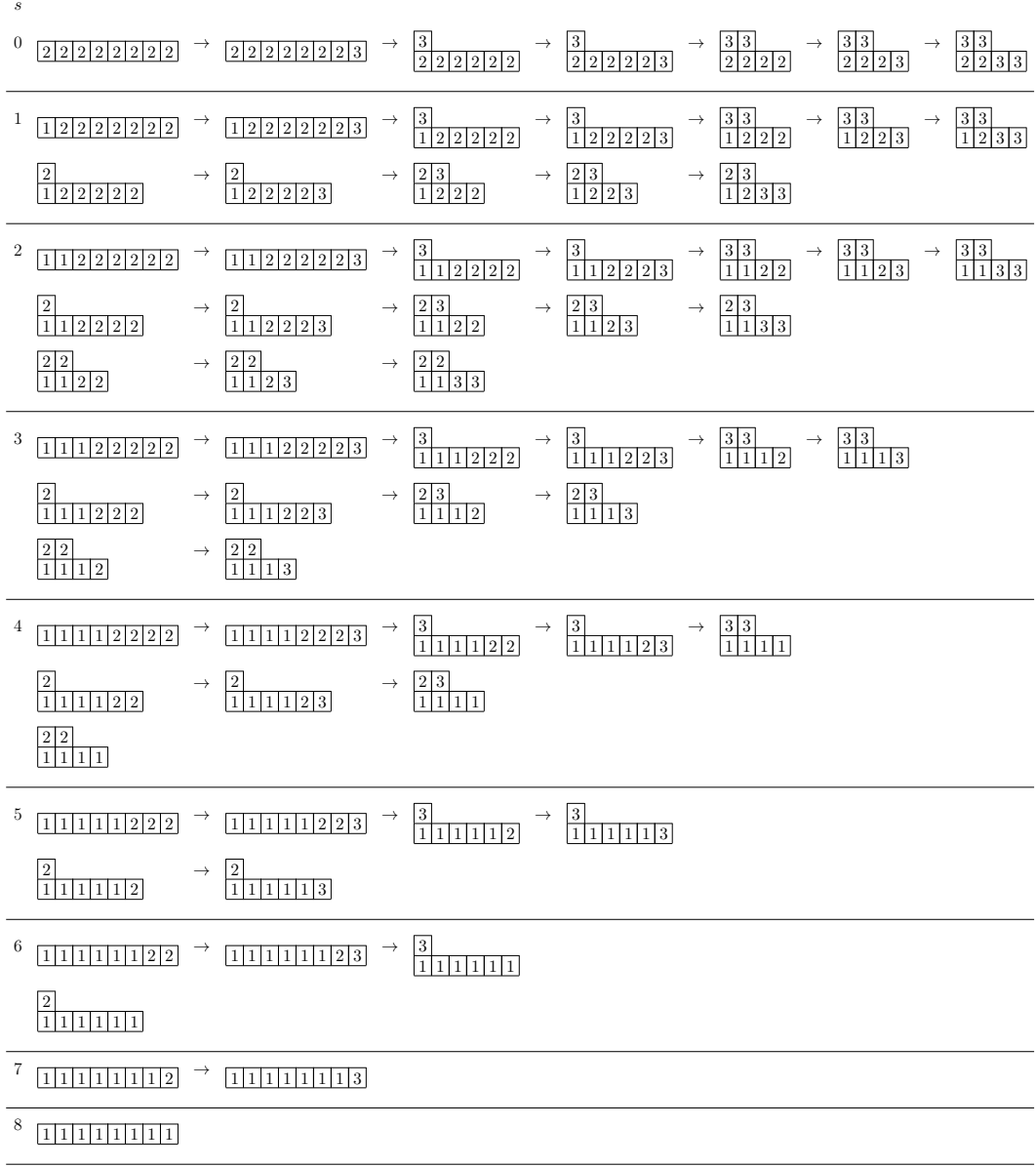}
}
\caption{New tableaux of $\AB_{6,2}$ generated through the action of $\psi$.}
\label{fig:ex-orbit}
\end{figure}

\begin{proposition}
\label{prop:psi-orbit}
Let $\pi' \in R'$ and let $\ABtwo$ be the subset of $\AB_{m,n}$ of tableaux with no label~$3$, then there exists a unique tableau $\pi \in \ABtwo$ such that $\pi'$ is the image of $\pi$ through multiple applications of $\psi$. More precisely, we have $\pi' = \psi^{(k)}(\pi)$ with $k = r(\pi) - r(\pi')$.
\end{proposition}

This is illustrated on Figure~\ref{fig:ex-orbit}: we see that all new tableaux of $\AB_{6,2}$ can be obtained by applying $\psi$ a certain amount of time on a tableau containing only $1$ and $2$. The tableaux are ordered by the number of $1$ they contain.

\begin{proof}
Note first that $\NewB_{m,n} \setminus R' = \ABtwo$. Indeed, if a tableau does not contain any $3$, it is necessarily a new tableau. 

Now remember that by Proposition~\ref{prop:psi-bij}, $\psi$ is a bijection between $R$ and $R'$. So there is a unique tableau $\pi = \psi'(\pi')$ in $R$ such that $\pi' = \psi(\pi)$. Besides, by Proposition~\ref{prop:psi} $r(\pi) = r(\pi') + 1$. If $\pi \not\in R'$, then $\pi \in \ABtwo$ and the result is proved. Otherwise, we can apply $\psi'$ again until we reach an element outside of $R'$ : the process necessarily ends because the number of raisable cells is increased at each step. Beside the number of times $\psi$ needs to be applied is given by difference in the number of raisable cells.
\end{proof}

\begin{corollary}
We have
\begin{equation}
\label{eq:amn-orbit}
A_{m,n}(q,t,1) - A_{m-1,n-1}(q,t,1) = \sum_{\pi \in \ABtwo} q^{\pi_1} \left( \sum_{k=0}^{r(\pi)} t^{\pi_2 - k} \right)
\end{equation}
\end{corollary}

\begin{proof}
Thanks to Proposition~\ref{prop:new}, we have that

\begin{equation}
A_{m,n}(q,t,1) - A_{m-1,n-1}(q,t,1) = \sum_{\pi \in \NewB_{m,n}} q^{\pi_1} t^{\pi_2}.
\end{equation}

If $\pi \in \NewB_{m,n} \setminus R'$, then $\pi \in \AB_{m,n}^{|1,2|}$ and the term $q^{\pi_1} t^{\pi_2}$ corresponds to the $k=0$ term in~\eqref{eq:amn-orbit}.

Now let $\pi' \in R'$, \emph{i.e.}, $\pi' \in \NewB_{m,n}$ and it contains at least one cell labeled by $3$. The monomial corresponding to $\pi'$ in the sum is $q^{\pi'_1}t^{\pi'_2}$. By Proposition~\ref{prop:psi-orbit}, there is a unique $\pi \in \ABtwo$ such that $\pi' = \psi^{(k)}(\pi)$ with $k = r(\pi) - r(\pi')$. In particular $1 \leq k \leq r(\pi)$. As the number of $1$ is not changed by $\psi$, we have $\pi'_1 = \pi_1$. The number of $2$ is reduced by $1$ at each application of $\psi$ so $\pi'_2 = \pi_2 - k$. The term $q^{\pi'_1}t^{\pi'_2}$ corresponds to the $k = r(\pi) - r(\pi')$ in the sum of~\eqref{eq:amn-orbit}.
\end{proof}

This last results tells us that to perform our computation, we need to enumerate the tableaux of $\ABtwo$. They are actually very easy to construct (as we have seen in Section~\ref{subsec:qt-enum}). We express the following Lemma.

\begin{lemma}
\label{lem:abtwo}
Let $m,n$ be a triangular partition and $d$ such that $0 \leq d \leq \min(n,m-n)$. Then, for all $s$ with $0 \leq s \leq m+n$, there exists a tableau $\pi$ in $\ABtwo$ of shape $m+n-2d, d$ with $\pi_1 = s$ if and only if $d \leq s \leq m+n-2d$. If it exists, this is the only tableau of $\ABtwo$ with this shape and we have $r(\pi) = m+n-2d - \max(s,n)$ and $\pi_2 = m+n-d -s$. 
\end{lemma}

\begin{proof}
The condition of a semi-standard Young tableau makes it clear than once the shape is fixed, there is at most one possible tableau $\pi$ with a given number of $1$: writes the $s$ values $1$ on the lower row and fill the rest of the tableau with the value $2$. For this construction to be possible, the lower row needs to be at least of size $s$, which gives $s \leq m+n-2d$. Besides, for the tableau to be semi-stadard, each cell of the upper row needs to be above a cell labeled by $1$: this gives $d \leq s$. 

The raisable cells are the cells of the lower row after the first $n$ cells and labeled by $2$. To count them, we thus need to remove either $n$ or $s$ to the length of the row, which gives $r(\pi) = m+n-2d-\max(n,s)$. 

The number of cells labeled by $2$ is the total number of cells minus $s$, \emph{i.e.}, $\pi_2 = m+n-2d+d-s = m+n-d-s$
\end{proof}

We are almost ready to prove our main equality but we first need to state a small general enumeration trick which we will be using multiple times.

\begin{lemma}
\label{lem:trick}
Let $a,b \in \NN$ with $a \leq b$ then

\begin{equation}
\sum_{i=0}^a \sum_{j=0}^b t^{i+j} = \sum_{i=0}^a \sum_{j=0}^{b+a-2i} t^{i+j}.
\end{equation}
\end{lemma}

\begin{proof}
Note that as $a \leq b$, the sum between $j=0$ and $j=b+a-2i$ is well defined with at least one term because $b+a-2i \geq b-a \geq 0$. This allows for classical summation of geometric series and we find that both expressions are equal to
\begin{equation}
\frac{1 - t^{a+1} - t^{b+1} + t^{a+b+2}}{(1-t)^2}.
\end{equation}
\end{proof}

\begin{remark}
Beside the easy computational explanation, there is also a nice combinatorial interpretation of these sums as we illustrate on Figure~\ref{fig:sum-trick}. They both enumerate entries in a $2$ dimensional array. The first sum does a classical enumerates line by line while the second sum enumerates first a ribbon going through the first line and last column, then the ribbon underneath and so on.
\end{remark}

\begin{figure}[ht]
\input{figures/sum-trick}
\caption{Illustration of two sum enumerations.}
\label{fig:sum-trick}
\end{figure}
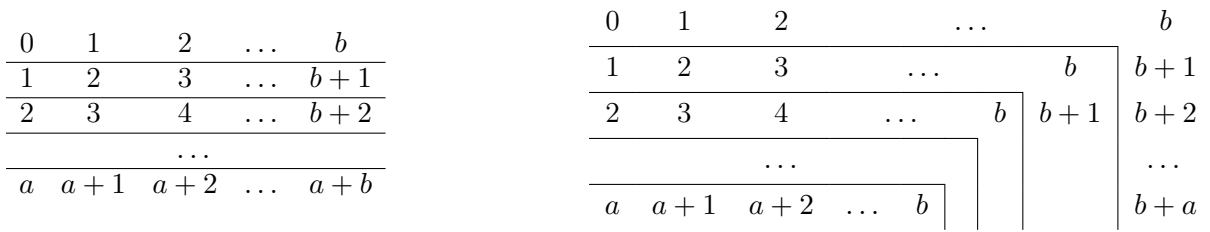

\begin{proposition}
\label{prop:lattice-rec-equal}
Let $(m,n)$ be a triangular partition with $n > 0$, then 

\begin{equation}
A_{m,n}(q,t,1) - A_{m-1,n-1}(q,t,1) = P_{m,n}(q,t) - P_{m-1,n-1}(q,t).
\end{equation}

\end{proposition}

\begin{proof}
Following the formula of Proposition~\ref{prop:latice-rec-diff}, we are going to split the sum~\eqref{eq:amn-orbit} depending on the power of $q$ and show that we get an equality in each case. The power of $q$ corresponds to the number of $1$ in the tableaux. It varies between $0$ and $m+n$. We can split the proof into three cases following the $3$ possibilities of Proposition~\ref{prop:nu-tam-sim}.

\paragraph{Case $0 \leq s \leq m -n$.} Let us fix the power of $q$ to a certain value $s$ with $0 \leq s \leq m-n$. The coefficient of $q^s$ is a polynomial in $t$. In $P_{m,n} - P_{m-1,n-1}$ it is given in~\eqref{eq:lattice-rec-diff-small}. In $A_{m,n}(q,t,1) - A_{m-1,n-1}(q,t,1)$, it is given by an enumeration on tableaux $\pi$ such that $\pi_1 = s$. We thus need to prove

\begin{align}
\sum_{\substack{\pi \in \ABtwo \\ \pi_1 = s}} \sum_{k=0}^{r(\pi)} t^{\pi_2 - k}
&=
\sum_{k=0}^{\min(s,n)} \sum_{v=n-k}^{m+n-k-s}t^v. \\
&= t^{n-\min(s,n)} \sum_{i=0}^{\min(s,n)} \sum_{j=0}^{m-s} t^{i+j}.
\end{align}

The second line is obtained by a simple change of variable $i = \min(s,n) - k$ to get something in the form of Lemma~\ref{lem:trick}. Note that by hypothesis, we have $s \leq m- n$ which implies $n \leq m -s$ and so $\min(s,n) \leq m-s$, so Lemma~\ref{lem:trick} can be applied with $a=\min(s,n)$ and $b=m-s$.

Now, let us identify the tableaux of $\ABtwo$ such that $\pi_1 = s$. Using Lemma~\ref{lem:abtwo}, there is one such tableau for each $d$ with $0 \leq d \leq \min(n,m-n)$ and $d \leq s \leq m+n-2d$. If $d \leq \min(n,m-n)$, in particular $d\leq n$ and we have $m+n-2d \geq m-n \geq s$ by hypothesis on $s$ so this one condition is always satisfied.

We get that for each $d$, with $0 \leq d \leq \min(s,n,m-n)$, we have exactly one tableau $\pi$ in $\ABtwo$ of shape $m+n-2d,d$ with $\pi_1 = s$. As by hypothesis $s \leq m-n$, we actually get $0 \leq d \leq \min(s,n)$. Using the values of $r(\pi)$ and $\pi_2$ given in Lemma~\ref{lem:abtwo}, we obtain 

\begin{align}
\sum_{\substack{\pi \in \ABtwo \\ \pi_1 = s}} \sum_{k=0}^{r(\pi)} t^{\pi_2 - k}
&=
\sum_{d=0}^{\min(s,n)} \sum_{k=0}^{m+n-2d-\max(s,n)} t^{m+n-d-s-k} \\
&=
\sum_{d=0}^{\min(s,n)} \sum_{v=\max(s,n) - s}^{m+n-s-2d} t^{d+v}.
\end{align}
by taking $v = m+n-2d-s-k$. Now see that $n - \min(s,n) = \max(s,n) - s$ which allows us to rewrite the sum as

\begin{align}
\sum_{\substack{\pi \in \ABtwo \\ \pi_1 = s}} \sum_{k=0}^{r(\pi)} t^{\pi_2 - k}
&=
t^{n-\min(s,n)} \sum_{i=0}^{\min(s,n)} \sum_{j=0}^{m-s+\min(s,n)-2d} t^{i+j}.
\end{align}
We conclude using Lemma~\ref{lem:trick} with $a = \min(s,n)$ and $b = m -s$. Note that this enumeration is illustrated on an example on Figure~\ref{fig:ex-orbit} for $m=6$ and $n=2$ (lines $s=0$ to $s=4$). 

\paragraph{Case $s > m-n$ and of same parity as $m+n$.} As we have seen in Proposition~\ref{prop:nu-tam-sim}, then $s = m+n-2i$ with $0 \leq i \leq n-1$. The coefficient of $q^s$ in $P_{m,n} - P_{m-1,n-1}$ is a polynomial in $t$ given in~\eqref{eq:lattice-rec-diff-even}. We need to prove

\begin{align}
\sum_{\substack{\pi \in \ABtwo \\ \pi_1 = s}} \sum_{k=0}^{r(\pi)} t^{\pi_2 - k}
&=
\sum_{j=0}^{i} \sum_{d=j}^{i+j}t^d = \sum_{j=0}^{i} \sum_{v=0}^{i} t^{j+v}.
\end{align}

Let us identify the tableaux of $\ABtwo$ such that $\pi_1 = s$. Again, we use Lemma~\ref{lem:abtwo}. The condition $s \leq m+n-2d$ gives  $d \leq i$. The other conditions $d \leq \min(n,m-n)$ and $d \leq s$ are then always satisfied. Indeed, as $i < n$, we have in particular $i \leq m-n$ (because $m,n$ is triangular so by Remark~\ref{rem:triang}, $n\leq m-n+1$) so $i \leq \min(n,m-n)$. If $d \leq i$, this implies $d \leq \min(n,m-n)$. Besides, by Remark~\ref{rem:triang}, $i,m+n-2i$ is a proper partition meaning $i \leq m+n-2i$. Then $d \leq i$ also implies $d \leq m+n-2i$. 

Now, as $i < n$ and $i \leq m-n$, we have $m+n-2i \geq m+n - n - (m-n) = n$ so $\max(s,n) = s$. By Lemma~\ref{lem:abtwo}, we have

\begin{align}
r(\pi) &= m+n-2d - s = 2i - 2d, \\ 
\pi_2 &= m+n-d-s = 2i - d,
\end{align}
and obtain

\begin{equation}
\sum_{\substack{\pi \in \ABtwo \\ \pi_1 = s}} \sum_{k=0}^{r(\pi)} t^{\pi_2 - k}
= 
\sum_{d=0}^i \sum_{k=0}^{2i - 2d} t^{2i-d-k} = \sum_{d=0}^i \sum_{v=0}^{2i-2d} t^{d+v}.
\end{equation}

We conclude again using Lemma~\ref{lem:trick}. This enumeration is illustrated on an example on Figure~\ref{fig:ex-orbit} for $m=6$ and $n=2$ (lines $s=6$ and $s=8$).

\paragraph{Case $s > m-n$ and of different parity than $m+n$.} As we have seen in Proposition~\ref{prop:nu-tam-sim}, then $s = m+n-2j +1$ with $1 \leq j \leq n$. The coefficient of $q^s$ in $P_{m,n} - P_{m-1,n-1}$ is a polynomial in $t$ given in~\eqref{eq:lattice-rec-diff-odd}. We need to prove

\begin{align}
\sum_{\substack{\pi \in \ABtwo \\ \pi_1 = s}} \sum_{k=0}^{r(\pi)} t^{\pi_2 - k}
&=
\sum_{i=0}^{j-1} \sum_{d=i}^{i+j}t^d = \sum_{i=0}^{j-1} \sum_{v=0}^{j} t^{i+v}.
\end{align}

Let us identify the tableaux of $\ABtwo$ such that $\pi_1 = s$. Again, we use Lemma~\ref{lem:abtwo}. The condition $s \leq m+n-2d$ gives  $2d +1 \leq 2j$ which is equivalent to $d < j$. Again, the other conditions $d \leq \min(n,m-n)$ and $d \leq s$ are then always satisfied. Remember that by Remark~\ref{rem:triang}, $n \leq m-n+1$. As we have $d < n$, we also have $d \leq m-n$ and so $d\leq \min(n,m-n)$. Similarly, we have $j \leq n \leq m-n+1$, so $m+n-2j+1 \geq m+n -n -(m-n+1) = n-1 \geq d$. 

Besides, we also have $m+n-2j+1 \geq m+n -2n + 1 = m-n+1 \geq n$ so $\max(s,n) = s$ and we have

\begin{align}
r(\pi) &= m+n-2d - s = 2j - 2d - 1, \\ 
\pi_2 &= m+n-d-s = 2j - d - 1,
\end{align}
and obtain

\begin{equation}
\sum_{\substack{\pi \in \ABtwo \\ \pi_1 = s}} \sum_{k=0}^{r(\pi)} t^{\pi_2 - k}
= 
\sum_{d=0}^{j-1} \sum_{k=0}^{2j - 2d - 1} t^{2j-d-1-k} = \sum_{d=0}^{j-1} \sum_{v=0}^{2j-2d-1} t^{d+v}.
\end{equation}

We conclude again using Lemma~\ref{lem:trick}.  This enumeration is illustrated on an example on Figure~\ref{fig:ex-orbit} for $m=6$ and $n=2$ (lines $s=5$ and $s=7$).
\end{proof}

\begin{proof}[Proof of Theorem~\ref{thm:qtrsym}]
This is direct by induction on $n$ using Proposition~\ref{prop:lattice-rec-0} for the initial case and Proposition~\ref{prop:lattice-rec-equal} for the induction.
\end{proof}

\begin{remark}
Note that even though we prove that both expressions are equal, we have to use some enumeration trick and do not exhibit a direct bijection between intervals and semi-standard tableaux. Indeed, we have found that the ``natural''way to enumerate intervals is to fix the maximal element. But this enumeration does not translate well into semi-standard tableaux which is why we have to use Lemma~\ref{lem:trick} to prove the equality. In the end, it is not clear which tableau would ``naturally'' correspond to which interval. This explains why we have no interpretation of the parameter $r$ (the number of values equal to $3$ in the tableaux) on the interval size.
\end{remark}

\section{Conclusion}

To our knowledge, Theorem~\ref{thm:qtrsym} is the first to formally exhibit a link between enumeration of intervals in Tamari-like lattices and Schur functions as all results in the general case (even for the classical Tamari lattice) are only conjectural. For this reason, we believe it is an important result and more so as it opens to Conjecture~\ref{conj:lattice}. Nevertheless, we see that even in the ``simple'' case of $2$-partitions, the proof of the equality is rather technical and does not rely on a bijection as we have explained in our previous remark.

\bibliographystyle{alphaurl} 
\bibliography{bibli} 

\appendix

\begin{table}
\input{figures/al_table}
\caption{List of $A_\lambda$ polynomials checked against Conjecture~\ref{conj:lattice}}
\label{tab:al-table}
\end{table}

\end{document}

%% file: figures/sum-trick.tex
\begin{subfigure}{.49 \textwidth}
\begin{tabular}{ccccc}
0 & 1 & 2 & $\dots$ & $b$ \\ \hline
1 & 2 & 3 & $\dots$ & $b+1$ \\ \hline
2 & 3 & 4 & $\dots$ & $b+2$ \\ \hline
\multicolumn{5}{c}{$\dots$} \\ \hline
$a$ & $a+1$ & $a+2$ & $\dots $ & $a+b$
\end{tabular}
\end{subfigure}
\begin{subfigure}{.49 \textwidth}
\begin{tblr}{
	colspec = {ccccccccc},
	cell{1}{4} = {c=5}{c},
	cell{2}{4} = {c=4}{c},
	cell{3}{4} = {c=3}{c},
	cell{4}{1} = {c=6}{c},
	hspan = even,
	hline{2} = {1-8}{},
	vline{9} = {2-6}{},
	hline{3} = {1-7}{},
	vline{8} = {3-6}{},
	hline{4} = {1-6}{},
	vline{7} = {4-6}{},
	hline{5} = {1-5}{},
	vline{6} = {5-6}{}
	}
0       & 1     & 2     & $\dots$ &     & &     &       & $b$ \\
1       & 2     & 3     & $\dots$ &     & &     & $b$   & $b+1$ \\
2       & 3     & 4     & $\dots$ &     & & $b$ & $b+1$ & $b+2$ \\
$\dots$ &       &       &         &     & &     &       & $\dots$ \\
$a$     & $a+1$ & $a+2$ & $\dots$ & $b$ & &     &       & $b+a$  
\end{tblr}	
\end{subfigure}